\documentclass[11pt]{amsart}
\usepackage{times,amsmath}
\usepackage{amsthm}
\usepackage{amsmath,algorithm2e,bm}
\usepackage{amsmath,amssymb,amsfonts,amscd}
\usepackage{graphicx,latexsym,color}
\usepackage{hyperref}
\usepackage{multirow}
\usepackage{floatrow}
\usepackage{wrapfig}

\usepackage{verbatim}
\usepackage{algorithmic}
\usepackage{algorithm2e}
\usepackage[sort&compress,numbers]{natbib}

\allowdisplaybreaks[4]

\usepackage{subfig}

\newcommand{\p}{\partial}

\newcommand{\dif}{\,\mathrm{d}}

\newtheorem{thm}{Theorem}[section]

\def \bes{\begin{eqnarray}}
\def \ees{\end{eqnarray}}
\def \bns{\begin{eqnarray*}}
\def \ens{\end{eqnarray*}}

\newenvironment{eqa}{\begin{equation}%
  \begin{array}{rcl}}{\end{array}\end{equation}}
\newcommand\beqa{\begin{eqa}}
\newcommand\eeqa{\end{eqa}}
\newcommand{\re}[1]{\mbox{$($\ref{#1}$)$}}\baselineskip 1pt

\usepackage[bottom=1.2in,top=1in,left=1in,right=1in]{geometry}

\begin{document}
\bibliographystyle{plain}
\numberwithin{equation}{section}

\title{Optimal control of treatment in a free boundary problem modeling multilayered tumor growth}

\author{Xinyue Evelyn Zhao}  
\address{Department of Mathematics, University of Tennessee, Knoxville, TN 37996, USA} 
\email{xzhao45@utk.edu}

\author{Yixiang Wu}
\address{Department of Mathematical Sciences, Middle Tennessee State University, TN 37132, USA}
\email{yixiang.wu@mtsu.edu}

\author{Rachel Leander}
\address{Department of Mathematical Sciences, Middle Tennessee State University, TN 37132, USA}
\email{rachel.leander@mtsu.edu}

\author{Wandi Ding}
\address{Department of Mathematical Sciences, Middle Tennessee State University, TN 37132, USA}
\email{wandi.ding@mtsu.edu}

\author{Suzanne Lenhart }
\address{Department of Mathematics, University of Tennessee, Knoxville, TN 37996, USA} 
\email{slenhart@utk.edu}

\begin{abstract}
We study the optimal control problem of a free boundary PDE model describing the growth of multilayered tumor tissue in vitro. We seek the optimal amount of tumor growth inhibitor that simultaneously minimizes the thickness of the tumor tissue and mitigates side effects. The existence of an optimal control is established, and the uniqueness and characterization of the optimal control are investigated. Numerical simulations are presented for some scenarios, including the steady-state  and parabolic cases.

\vspace{0.1in}
\end{abstract}

\maketitle

\section{Introduction}
Over the past few decades, mathematical models that describe solid tumor growth as free boundary problems have been proposed and studied; for a comprehensive review, see \cite{bazaliy2003global,bazaliy2003free,cui2009well,fontelos2003symmetry,friedman2007mathematical,friedman2008multiscale,friedman2009free,friedman2006bifurcation,friedman2006asymptotic,friedman2008stability,friedman1999analysis,friedman2001symmetry} and references therein. These models, which consider the tumor tissue as a density of proliferating cells, are based on reaction-diffusion process and conservation of mass. Within this framework, the effects of various factors such as inhibitors \cite{wang2014bifurcation,wu2012bifurcation}, angiogenesis \cite{huang2017bifurcation,huang2021asymptotic,zhou2022effect,friedman2015analysis}, time delays \cite{zhao2020impact,zhao2020symmetry,he2022linear}, and necrosis \cite{lu2023bifurcation,lu2022nonlinear,hao2012continuation,cui2001analysis} on tumor growth have also been extensively studied.

However, optimizing treatment strategies within these models remains a challenging and relatively unexplored area. While the theory of optimal control is well-established for ordinary differential equations (ODEs) \cite{kirschner1997optimal,lenhart2007optimal,jung2002optimal,miller2010modeling,kim2017mathematical,ledzewicz2019optimal,schattler2015optimal} and partial differential equations (PDEs) in fixed domains \cite{bintz2020optimal,yousefnezhad2021optimal,casas2017optimal,valega2023resource,li2012optimal}, 
the introduction of an unknown boundary in free boundary problems adds significant complexity. There have been some notable contributions to the optimal control of free boundary problems, such as the obstacle problem \cite{friedman1986optimal,friedman1987optimal,barbu1984optimal} and the Stefan problem \cite{abdulla2019optimal,abdulla2020optimal,abdulla2021optimal,hinze2007optimal}. In 2013, the authors in \cite{calzada2013optimal} conducted numerical experiments and used real data to find the optimal cancer therapy in a free boundary PDE model of tumor growth. However, to the best of the authors' knowledge, no previous work has theoretically addressed the optimal control of tumor growth models with free boundaries.

This study seeks to fill this gap by proposing an optimal control framework for a free boundary problem modeling the growth of a multilayer tumor. A multilayer tumor is a cluster of tumor cells cultivated in laboratory settings using advanced tissue culture techniques \cite{kyle1999characterization,kim2004three,mueller1997three}. It consists of many layers of tumor cells so that it has an observable thickness. The geometric configuration of the problem is 
$$
\Omega(t) = \{(\bm{x},y)\in \mathbb{R}^2 \times \mathbb{R};\;  0<y<\rho(t,\bm{x})\},
$$
where $\rho(t,\bm{x})$ is an unknown positive function. The tumor receives a constant supply of nutrients (e.g., oxygen or glucose) through the upper boundary 
$$
\Gamma(t)=\{(\bm x, y);\;  \bm x\in \mathbb{R}^2;\ y=\rho(t,\bm{x})\},
$$
which depends on the time variable $t$. The nutrients then diffuse into all parts of the tumor, enabling tumor cells to live and proliferate. The lower boundary $\Gamma_0: y=0$ is fixed and is assumed to be an impermeable support membrane, so that neither nutrient nor tumor cells can pass through it. In the model, there are three unknown functions:
\begin{itemize}
    \item the free boundary $\rho = \rho(t,\bm{x})$; 
    \item the concentration of nutrients $\sigma = \sigma(\bm{x},y,t)$;
    \item the pressure within the tumor $p=p(\bm{x},y,t)$, which is related to the velocity $\vec{V}$ of the cells.
\end{itemize}
We assume that the density of tumor cells depends linearly on the concentration of nutrients; for simplicity, we identify the density of tumor cells directly with $\sigma$. As in \cite{cui2009well,zhou2008bifurcation}, $\sigma$ and $p$ satisfy the following equations:
\begin{eqnarray}
   &\label{e1}\sigma_t = \Delta \sigma - \sigma  &\hspace{2em} (\bm{x},y)\in \Omega(t), \;0<t<T,\\
    &\label{e2}\sigma = 1 &\hspace{2em} (\bm{x},y)\in \Gamma(t), \;0<t<T,\\
    &\label{e3}\frac{\partial \sigma}{\partial y} = 0 &\hspace{2em} (\bm{x},y)\in \Gamma_0, \;0<t<T,\\
    &\label{e4}-\Delta p = \mu(\sigma - \widetilde{\sigma}) &\hspace{2em} (\bm{x},y)\in \Omega(t), \;0<t<T,\\
    &\label{e5}p = \gamma \kappa &\hspace{2em} (\bm{x},y)\in \Gamma(t), \;0<t<T,\\
    &\label{e6}\frac{\p p}{\p y} = 0 &\hspace{2em} (\bm{x},y)\in \Gamma_0, \;0<t<T,
\end{eqnarray}
where $0<\widetilde{\sigma}<1$ is a threshold concentration for the tumor to sustain itself, and $\mu$ is a parameter expressing the ``intensity" of tumor expansion due to mitosis (if $\sigma > \widetilde{\sigma})$ or tumor shrinkage by apoptosis (if $\sigma < \widetilde{\sigma})$. In the model, it is assumed that the tumor region is a porous medium, so that Darcy's law $\vec{V}=-\nabla p$ holds. Combined with the law of conservation of mass $\text{div} \vec{V} = \mu(\sigma-\widetilde{\sigma})$, this leads to the equation \re{e4}. The boundary condition \re{e5} is due to cell-to-cell adhesiveness, where $\gamma$ represents the surface tension coefficient and 
{\color{black}$$\kappa = \frac12 \text{div}\Big(\frac{\nabla(\rho-y)}{|\nabla(\rho-y)|^2}\Big)$$
denotes the mean curvature of the free boundary $y=\rho(t,\bm{x})$}. Furthermore, the continuity of the velocity field $\vec{V}$ on the free boundary gives
\begin{equation}
    \label{e7}V_n = \vec{V}\cdot \vec{n} =  -\nabla p \cdot \vec{n} = -\frac{\p p}{\p n} \hspace{2em} (\bm{x},y)\in \Gamma(t), \;0<t<T,
\end{equation}
where $\frac{\p}{\p n}$ is the derivative along the outward normal $\vec{n}$ and $V_n$ is the velocity of the free boundary $\p \Omega(t)$ in the outward normal direction $\vec{n}$, so that the velocity of the free-boundary is the velocity of cells at the free boundary. The following initial conditions are also prescribed:
\begin{eqnarray}
    &\label{e8} \sigma(\bm{x},y,0) = \sigma_0(\bm{x},y) \hspace{2em} (\bm{x},y)\in \Omega(0),\\
    &\label{e9} \Omega(0) = \{(\bm{x},y)\in \mathbb{R}^2 \times \mathbb{R}; \; 0<y<\rho_0(\bm{x})\}.
\end{eqnarray}

The local well-posedness, asymptotic stability, and bifurcation solutions of the model \re{e1} -- \re{e9} were investigated in \cite{cui2009well,zhou2008bifurcation}. In \cite{cui2009well}, the authors also considered the flat solutions, where all the variables are independent of $\bm{x}$, such that $\rho = \rho(t)$, $\sigma=\sigma(y,t)$, $p=p(y,t)$. In this special case, $V_n = \rho'(t)$, $\frac{\p p}{\p n} = \frac{\p p}{\p y}$, and \re{e7} becomes
\begin{equation}\label{e10}
    \rho'(t) = -\frac{\p p}{\p y}\bigg|_{y=\rho(t)}.
\end{equation}
Integrating \re{e4} and utilizing \re{e6} and \re{e10} leads to
\begin{equation}\label{e11}
    \int_0^{\rho(t)} \mu(\sigma-\widetilde{\sigma}) \dif y = \int_0^{\rho(t)} -p_{yy} \dif y = -p_y(\rho(t),t) + p_y(0,t) = \rho'(t).
\end{equation}
Thus, we derive the equation for $\rho$. Note that \re{e11} does not include $p$, allowing us to eliminate the equations for $p$ and reduce the problem to finding $(\sigma,\rho)$ such that:
\begin{eqnarray}
    \label{fp1}&\displaystyle\sigma_t = \sigma_{yy} - \sigma &\hspace{2em} 0<y<\rho(t), \;0<t<T,\\
    \label{fp2}&\displaystyle\sigma =1 &\hspace{2em}y=\rho(t),\;0<t<T,\\
    \label{fp3}&\displaystyle\sigma_y = 0 &\hspace{2em} y = 0, \;0<t<T,\\
    \label{fp4}&\displaystyle \rho'(t) = \int_0^{\rho(t)} \mu (\sigma - \widetilde{\sigma})\dif y & \hspace{2em} 0<t<T,
\end{eqnarray}
with initial conditions
\begin{eqnarray}
\label{fp5}& \displaystyle \sigma(y,0) = \sigma_0(y)  \hspace{2em}0<y<\rho_0.
\end{eqnarray}
It was proved in \cite{cui2009well} that under the condition $0<\widetilde{\sigma}<1$, the system \re{fp1} -- \re{fp5} admits a unique stationary solution, and this solution is an asymptotically stable equilibrium of the model \re{e1} -- \re{e9} if the surface tension coefficient is large enough. 

In this paper, we introduce a control function in the system \re{fp1} -- \re{fp5}, where the control $m(t)$ represents a tumor growth inhibitor. This inhibitor acts by reducing nutrient levels inside the tumor \cite{byrne1995growth}. The state system with control is defined as:
\begin{eqnarray}
    &\label{ce1}\displaystyle\sigma_t = \sigma_{yy} - \sigma - m\sigma &\hspace{2em} 0<y<\rho(t), \;0<t<T,\\
    &\label{ce2}\displaystyle\sigma =1 &\hspace{2em}y=\rho(t),\;0<t<T,\\
    &\label{ce3}\displaystyle\sigma_y=0 &\hspace{2em} y = 0, \;0<t<T,\\
    &\label{ce4}\displaystyle\rho'(t) = \int_0^{\rho(t)} \mu (\sigma - \widetilde{\sigma})\dif y & \hspace{2em} 0<t<T,\\
    &\label{ce5}\displaystyle\sigma(y,0)=\sigma_0(y)& \hspace{2em} 0<y<\rho_0.
\end{eqnarray}
Our goal is to identify the optimal control within the admissible control set
\begin{equation}
    \label{set}
    U_M = \{ m\in L^\infty (0,T);\; 0\le m(t)\le M, t\in [0, T]\}
\end{equation}
which minimizes the following objective functional:
\begin{equation}
    J(m) = \int_0^T\big( \rho(t)+B m^2(t) \big)\dif t,\label{obj}
\end{equation}
where $(\sigma, \rho)$ is the solution of 
\eqref{ce1}-\eqref{ce5}.
Here, we focus on a spatially independent control function. {\color{black}This simplification assumes that the inhibitor concentration is uniform throughout the tumor tissue.}
The objective functional \re{obj} is designed to balance the combined goals of controlling tumor size and minimizing side effects. We assume that the side effects ``cost'' is a quadratic function of $m$. The coefficient $B\ge 0$ represents a balancing parameter between the two goals.

The structure of this paper is as follows: Section 2 explores the steady-state solution of the control problem \re{ce1} -- \re{obj} as a special case, which is solved using calculus techniques. Section 3 is dedicated to establishing the existence and uniqueness of a positive solution for the state system. Section 4 demonstrates the existence of an optimal control. In Section 5, we derive the sensitivity system, the adjoint system, and the optimality system. We present the numerical algorithm and simulation results in Section 6, and provide a discussion in Section 7.

\section{A special case: steady-state solution}
In the consideration of the steady-state case, the control $m(t)=m$ is a constant and the solution to the state system $(\sigma,\rho)$ is independent of the time variable $t$. Consequently, the state system \re{ce1} -- \re{ce3} is reduced to 
\begin{eqnarray}
    &\label{sce1}\displaystyle\sigma_{yy} = (1+m)\sigma &\hspace{2em}0<y<\rho,\\
    &\label{sce2}\displaystyle\sigma=1 &\hspace{2em} y=\rho,\\
    &\label{sce3}\displaystyle\sigma_y=0&\hspace{2em} y=0,\\
    &\label{sce4}\displaystyle \int_0^{\rho} \mu (\sigma-\widetilde{\sigma})\dif y = 0.
\end{eqnarray}
The  first three equations of the system can be explicitly solved to obtain
\begin{equation}
    \label{ss-sol}
    \sigma(y)=\frac{\cosh(\sqrt{1+m}y)}{\cosh(\sqrt{1+m}\rho)},
\end{equation}
where, by \re{sce4}, the thickness of the tumor tissue $\rho$ is uniquely determined by
\begin{equation}
\label{ss-solrho}
    \frac{1}{\sqrt{1+m}\rho}\tanh(\sqrt{1+m}\rho) = \widetilde{\sigma}.
\end{equation}
Correspondingly, the objective functional also simplifies, eliminating the dependency on the time variable,
\begin{equation}
    \label{ss-obj}
    J(m)=\rho + B m^2.
\end{equation}
The steady-state optimal control problem thus becomes:
\begin{eqnarray*}
    &\displaystyle\min\limits_{m\in[0,M]}\; J(m) = \rho + B m^2\\
    &\displaystyle\text{subject to } \frac{\tanh(\sqrt{1+m}\rho)}{\sqrt{1+m}\rho} = \widetilde{\sigma}.
\end{eqnarray*}

To analyze this optimization problem, we introduce $g(x)=\frac{\tanh(x)}{x}$. Note that $g(x)$ is monotonically decreasing for $x\ge0$. It follows from \re{ss-solrho} that, given $\widetilde{\sigma}>0$,
$$\sqrt{1+m}\rho = g^{-1}(\widetilde{\sigma}).$$
Thus,
\begin{equation}
\label{ss-deri}
\frac{\p \rho}{\p m} = -\frac12 \frac{g^{-1}(\widetilde{\sigma})}{(1+m)^{\frac32}}.
\end{equation}

To find the critical point of the objective functional $J(m)$, we set $J'(m)=\frac{\p \rho}{\p m} + 2B m = 0$. In conjunction with \re{ss-deri}, this analysis yields the equation of the optimal control:
\begin{equation}
    \label{ss-oc}
    m_* (1+m_*)^{\frac32} = \frac{g^{-1}(\widetilde{\sigma})}{4B}.
\end{equation}
{\color{black}Note that the left-hand side function $h(x)=x(1+x)^{\frac32}$ is monotonically increasing for $x\ge 0$, thus the equation \re{ss-oc} admits a unique solution $m_*$.}
Recalling the admissible set $U_M$, the optimal control is defined as
\begin{equation}
    \label{ss-ocs}
    m^*=\min\{m_*,M\}.
\end{equation}
Examination of \re{ss-oc} reveals that an increase in $B$ corresponds to a decrease in  $m_*$, which aligns intuitively with the premise that increasing the severity of side effects may justify the reduction or absence of medication. Moreover, a higher value of $\widetilde{\sigma}$ leads to a lower optimal control $m_*$. This relationship is explained by the role of $\widetilde{\sigma}$ as a critical concentration threshold of nutrients necessary for tumor sustainability and growth. When $\widetilde{\sigma}$ is larger, the tumor requires a higher concentration to sustain itself; consequently, a minimal control effort is sufficient to reduce the nutrient concentration below this critical threshold, thereby inhibiting tumor growth.

\section{Existence and uniqueness of a positive state solution}

In this section, we prove the existence and uniqueness of a solution to the state system \re{ce1} -- \re{ce5}. To analyze the system, we shall make a change of variables to flatten the free boundary. Let
$$\xi= \frac{y}{\rho(t)},\hspace{2em} \sigma(y,t) = u\Big(\frac{y}{\rho(t)},t\Big)=u(\xi,t),\hspace{2em} \sigma_0(y)=u_0\Big(\frac{y}{\rho(t)}\Big)=u_0(\xi).$$
With the change of variables, the problem \re{ce1} -- \re{ce5} becomes
\begin{eqnarray}
    &\label{eqn1} \displaystyle u_t - \frac{\rho'(t)}{\rho(t)}\xi u_{\xi} - \frac{1}{\rho^2(t)} u_{\xi\xi} = -u - mu &\hspace{2em} 0<\xi<1,\;0<t<T,\\
    &\label{eqn2}\displaystyle u = 1 &\hspace{2em} \xi=1, \; 0<t<T,\\
    &\label{eqn3}\displaystyle u_\xi = 0 &\hspace{2em} \xi=0,\; 0<t<T,\\
    &\label{eqn4}\displaystyle\rho'(t) = \rho(t) \int_0^1 \mu(u-\widetilde{\sigma})\dif \xi &\hspace{2em} 0<t<T,
\end{eqnarray}
supplemented with the initial conditions
\begin{eqnarray}
    &\label{eqn5}\displaystyle  \rho(0) = \rho_0 > 0,\\
    &\label{eqn6}\displaystyle u(\xi,0) = u_0(\xi)  \;\text{ for }\;0<\xi < 1,
\end{eqnarray}
which satisfy the compatibility conditions
\begin{equation}\label{compatible}
    u_0(1) = 1,\hspace{2em} (u_0)_\xi(0)=0.
\end{equation}
The existence and uniqueness of a solution are guaranteed by the following theorem. In all the proofs presented in this paper, we denote by $C$ a generic constant, which may vary from line to line. 

\begin{thm}\label{exi}
    Suppose that $u_0\in C^2[0,1]$ satisfy $0\le u_0(\xi)\le 1$ for $0<\xi < 1$ and the compatibility condition \re{compatible}.
  For any $m\in U_M$, system \re{eqn1} -- \re{eqn6} has a unique  solution $(u,\rho) \in W^{2,1,p}((0,1)\times(0,T)) \times C^1[0,T]$ for any $p>1$. In addition,
  \begin{eqnarray}
      \label{est:u}
      &0<u(\xi, t)<1 \quad \text{ for all } (\xi, t)\in (0,1)\times[0,T],\\
      \label{est:rho}&0< \rho_0 e^{-\mu\widetilde{\sigma} T} \le  \rho(t) \le \rho_0 e^{\mu(1-\widetilde{\sigma})T}\quad \text{ for all } t\in[0,T].
  \end{eqnarray}
\end{thm}
\begin{proof}
We shall use the contraction mapping principle to prove the existence and uniqueness of the solution. Take
\begin{equation}
    \label{K1bound}
    \mathcal{K}_1=\{s \in C^1[0,T];\;  s(0) = \rho_0, \|s\|_{C^1[0,T]}\le K_1\},
\end{equation}
where $K_1$ is a constant satisfying
\begin{equation}
    \label{K}
    K_1\ge \rho_0 e^{\mu(1-\widetilde{\sigma})}\big(1+\max\{\mu\widetilde{\sigma},\mu(1-\widetilde{\sigma})\}\big).
\end{equation}
By the definition of $\mathcal{K}_1$, we can choose $T<1$ sufficiently small such that $s(t)\ge \rho_0/2$ for all {\color{black} $s\in \mathcal{K}_1$ and} $t\in [0,T]$.

For each $s\in \mathcal{K}_1$, we first solve $u(\xi,t)$ from the following system:
\begin{equation}\label{thm1-1}
    \left\{
    \begin{aligned}
        &u_t - \frac{s'(t)}{s(t)}\xi u_{\xi} - \frac{1}{s^2(t)} u_{\xi\xi} + (1+m)u =0\hspace{2em} &0<\xi<1,\;0<t<T,\\
        & u = 1 \hspace{2em} &\xi = 1, \; 0<t<T,\\
        & u_\xi = 0 \hspace{2em} &\xi = 0, \; 0 <t<T,\\
        & u = u_0 \hspace{2em} &0<\xi<1,\; t=0.
    \end{aligned}
    \right.
\end{equation}
By the maximum principle,
\begin{equation}
    \label{est-u}
    0<u(\xi,t)<1\hspace{1em}\text{for }0<\xi<1,\;t>0.
\end{equation}
Since $m\in L^\infty$, applying $L^p$ estimates for parabolic equations 
and the embedding $W^{2,1,p}((0,1)\times(0,T))\hookrightarrow C^{1+\alpha,(1+\alpha)/2}([0,1]\times[0,T])$ ($p>n+2=3$, $\alpha=1-\frac{n+2}{p}=1-\frac{3}{p}$, Theorem 3.14 in \cite{hu2011blow}) yields
\begin{equation}\label{thm1-2}
    \|u\|_{C^{1+\alpha,(1+\alpha)/2}([0,1]\times[0,T])} \le C(T)\|u\|_{W^{2,1,p}((0,1)\times(0,T))}\le C(T),
\end{equation}
so, $u\in C^{1+\alpha,(1+\alpha)/2}( [0,1]\times[0,T])$. 
Note that the constant $C(T)$ in the embedding $W^{2,1,p} ((0,1)\times(0,T))\hookrightarrow C^{1+\alpha,(1+\alpha)/2}( [0,1]\times[0,T])$ depends on the lower bound of $T$. {\color{black}To avoid this dependency, we extend the function $s(t)$ so that it is defined in a fixed interval $[0,1]$. We let
\begin{equation*}
    \widetilde{s}(t) = \begin{cases}
        s(t), \quad \text{for } t\in [0,T],\\
        s(T), \quad \text{for } t\in (T,1].
    \end{cases}
\end{equation*}
It is clear that $\tilde{s}'(t)\equiv 0$ for $t\in(T,1)$, so,
\begin{equation*}
    \|\tilde{s}'\|_{L^\infty(0,1)} \le \|s'\|_{L^\infty(0,T)} \le \|s\|_{C^1[0,T]}.
\end{equation*}
In addition, 
$$\tilde{s}(t)\ge \rho_0/2 \quad \text{for } t\in[0,1]$$ 
and 
$$\|\tilde{s}\|_{L^\infty(0,1)} \le \|s\|_{L^\infty(0,T)} \le \|s\|_{C^1[0,T]}.$$ 
With $s$ replaced by $\tilde{s}$, we solve $u$ from the system \re{thm1-1} within the domain $[0,1]\times[0,1]$. Utilizing $L^\infty$ estimates and the embedding theorem, we have
\begin{equation}
    \|u\|_{C^{1+\alpha,(1+\alpha)/2}([0,1]\times[0,1])}\le C\|u\|_{W^{2,1,p}((0,1)\times(0,1))}\le C,
\end{equation}
thereby 
\begin{equation}
    \|u\|_{C^{1+\alpha,(1+\alpha)/2}([0,1]\times[0,
T])} \le C\|u\|_{C^{1+\alpha,(1+\alpha)/2}([0,1]\times[0,
1])}\le C,
\end{equation}
where the constants are now independent of $T$.}

Then, we define $\rho(t)$ as the solution of the following system,
\begin{equation}
    \label{thm1-3}
    \left\{
    \begin{aligned}
        & \rho'(t) = \rho(t) \int_0^1 \mu(u-\widetilde{\sigma})\dif \xi,\\
        &\rho(0) = \rho_0,
    \end{aligned}\right.
\end{equation}
which can be solved explicitly as
\begin{equation}
    \rho(t) = \rho_0 \exp\Big\{\int_0^t\int_0^1 \mu (u-\widetilde{\sigma})\dif \xi \dif s\Big\}.
\end{equation}
Obviously, $\rho(t)\in C^1[0,T]$. Utilizing the estimate of $u$ given in \re{est-u}, we derive the following estimates for $\rho(t)$:
\begin{equation}\label{rhot}
   \rho_0 e^{-\mu\widetilde{\sigma} t} \le  \rho(t) \le \rho_0 e^{\mu(1-\widetilde{\sigma})t}\hspace{2em}\text{for $0\le t\le T$.}
\end{equation}
Subsequently, by substituting \re{rhot} into \re{thm1-3}, we obtain bounds for $\rho'(t)$,
\begin{equation}\label{rhopt}
    -\mu\widetilde{\sigma}\rho(t) \le \rho'(t) \le \mu(1-\widetilde{\sigma})\rho(t)\hspace{2em}\text{for $0\le t\le T$.}
\end{equation}
For any $T<1$, we have 
$$\|\rho\|_{C^1[0,T]} = \|\rho\|_{L^\infty(0,T)} + \|\rho'\|_{L^\infty(0,T)} \le \rho_0 e^{\mu(1-\widetilde{\sigma})} + \rho_0 e^{\mu(1-\widetilde{\sigma})}\max\{\mu\widetilde{\sigma},\mu(1-\widetilde{\sigma})\}\le K_1$$
by the choice of $K_1$ in \re{K}. Thus, $\rho\in \mathcal{K}_1$. So, we can define a mapping $\mathcal{M}_1: \mathcal{K}_1\to \mathcal{K}_1$ by 
$$
\rho:=\mathcal{M}_1s.
$$


Next, we prove that $\mathcal{M}_1$ is a contraction mapping on $\mathcal{K}_1$. We take $s_1,s_2\in \mathcal{K}_1$ and denote by $u_1,u_2$ the corresponding solutions of system \re{thm1-1}. Denote the difference by $\phi = u_1-u_2$. The variable $\phi$ satisfies
\begin{equation}\label{thm1-ex}
    \left\{
    \begin{aligned}
        &\begin{aligned}\phi_t - \frac{s'_1(t)}{s_1(t)}\xi \phi_{\xi} - \frac{1}{s_1^2(t)} \phi_{\xi\xi} +(1+m)\phi = \Big(\frac{s_1'(t)}{s_1(t)}-\frac{s_2'(t)}{s_2(t)}\Big)\xi (u_2)_{\xi} \\
        \hspace{8em}+\Big(\frac{1}{s_1^2(t)}-\frac{1}{s_2^2(t)}\Big)(u_2)_{\xi\xi}\end{aligned}\hspace{2em} &0<\xi<1,\;0<t<T,\\
        & \phi= 0 \hspace{2em} &\xi = 1, \; 0<t<T,\\
        & \phi_\xi = 0 \hspace{2em} &\xi = 0, \; 0 <t<T,\\
        & \phi= 0 \hspace{2em} &0<\xi<1,\; t=0.
    \end{aligned}
    \right.
\end{equation}
Similar to the first part of this proof, we have
\begin{equation}
    \label{u1u2}
    0<u_1, u_2<1\hspace{1em}\text{for }0<\xi<1,\;t>0.
\end{equation}
Applying $L^p$ estimates on the system \re{thm1-ex} and the embedding theorem (also after artificially extending the time interval to $[0,1]$ as mentioned above), we obtain, for $0<T<1$,
\begin{equation}
    \label{thm1-6}
    \begin{split}
        \|u_1-u_2\|_{C^{1+\alpha,(1+\alpha)/2}([0,1]\times[0,T])} = 
    \|\phi\|_{C^{1+\alpha,(1+\alpha)/2}([0,1]\times[0,T])} &\le C\|\phi\|_{W^{2,1,p}((0,1)\times(0,1))}\\
    &\le C\|s_1-s_2\|_{W^{1,p}(0,1)} \\
    &\le C\|s_1-s_2\|_{C^1[0,T]},
    \end{split}
\end{equation}
where the constant $C$ is independent of $T$.
Then, we use system \re{thm1-3} to define $\rho_1$ and $\rho_2$, i.e.,
$$\rho_i =\mathcal{M}_1s_i \quad\text{ for } i=1,2.$$ 
Subtracting the equations for $\rho_1$ and $\rho_2$, we have, for $t>0$,
\begin{equation}\label{thm1-4}
    \begin{split}
        \big(\rho_1(t)-\rho_2(t)\big)'  &= \rho_1(t)\int_0^1 \mu(u_1-\widetilde{\sigma})\dif \xi - \rho_2(t)\int_0^1 \mu(u_2-\widetilde{\sigma})\dif \xi \\
        &= \big(\rho_1(t)-\rho_2(t)\big) \int_0^1 \mu(u_1-\widetilde{\sigma})\dif \xi + \rho_2(t)\int_0^1 \mu(u_1-u_2)\dif \xi.
    \end{split}
\end{equation}
Multiplying \re{thm1-4} with $\rho_1-\rho_2$ leads to
\begin{equation*}
\begin{split}
    \frac12 \frac{\dif }{\dif t}\big(\rho_1-\rho_2\big)^2 &= \big(\rho_1-\rho_2\big)^2 \int_0^1 \mu (u_1-\widetilde{\sigma})\dif \xi + \rho_2\big(\rho_1-\rho_2\big)\int_0^1\mu (u_1-u_2)\dif \xi\\
    &\le \mu(1-\widetilde{\sigma})\big(\rho_1-\rho_2\big)^2 + \frac12\bigg(\big(\rho_1-\rho_2\big)^2 + \rho_2^2\Big(\int_0^1 \mu(u_1-u_2)\dif \xi\Big)^2\bigg)\\
    &\le \frac{C_1}2|\rho_1-\rho_2|^2 + \frac{C_2}2\|u_1-u_2\|^2_{L^\infty((0,1)\times (0,T))},
\end{split}
\end{equation*}
for some constants $C_1,C_2>0$ that are independent of $T$. 
Utilizing the Gronwall inequality and noting that $(\rho_1-\rho_2)(0)=0$, we obtain
\begin{equation*}
    \big|(\rho_1-\rho_2)(t)\big|\le \sqrt{C_2 e^{C_1}}T^{1/2}\|u_1-u_2\|_{L^\infty((0,1)\times (0,T))}\hspace{2em}\text{for $0\le t\le T<1$}.
\end{equation*}
Combined with \re{thm1-6} and the initial condition $u_1(\xi,0)-u_2(\xi,0)=0$, it implies
\begin{equation*}
\begin{split}
    \big|(\rho_1-\rho_2)(t)\big| &\le \sqrt{C_2 e^{C_1}}T^{1/2} T^{(1+\alpha)/2} \|u_1 - u_2\|_{C^{1+\alpha,(1+\alpha)/2}([0,1]\times[0,T])}\\
    &\le C T \|s_1-s_2\|_{C^1[0,T]}\hspace{2em}\text{for $0\le t\le T<1$}.
    \end{split}
\end{equation*}
Similarly, we substitute this estimate back into \re{thm1-4} to obtain
\begin{equation*}
    \begin{split}
        \big|(\rho_1-\rho_2)'(t)\big| &\le C\|\rho_1-\rho_2\|_{L^\infty(0,T)} + C\|u_1-u_2\|_{L^\infty((0,1)\times(0,T))}\\
        &\le C T \|s_1-s_2\|_{C^1[0,T]} + C T^{(1+\alpha)/2}\|u_1 - u_2\|_{C^{1+\alpha,(1+\alpha)/2}([0,1]\times[0,T])}\\
        &\le C T^{1/2}  \|s_1-s_2\|_{C^1[0,T]}\hspace{2em}\text{for $0\le t\le T<1$}.
    \end{split}
\end{equation*}
When $T$ is sufficiently small, we have 
$$\|\mathcal{M}_1s_1-\mathcal{M}_1s_2\|_{C^1[0,T]} = \|\rho_1-\rho_2\|_{C^1[0,T]} \le \frac{1}{2} \|s_1-s_2\|_{C^1[0,T]},$$ 
which indicates that $\mathcal{M}_1$ is a contraction. By contraction mapping principle, the system admits a unique solution for small $T$.

{\color{black}Finally, we prove the solution can be extended to all $T>0$. Suppose, for the sake of contradiction, that there exists $0<\widetilde{T}<\infty$ such that $[0,\widetilde{T})$ is the maximum time interval for the existence of the solution. Notice that \re{est-u} does not depend on $\widetilde{T}$. Using \re{rhot} and \re{rhopt}, we obtain, for all $t\in[0,\widetilde{T})$,
\begin{equation}
    \label{wT1}
    \rho_0 e^{-\mu \widetilde{\sigma} \widetilde{T}} \le \rho(t) \le \rho_0 e^{\mu(1- \widetilde{\sigma}) \widetilde{T}},
\end{equation}
\begin{equation}
    \label{wT2}
    -\mu \widetilde{\sigma} \rho_0 e^{\mu(1- \widetilde{\sigma}) \widetilde{T}} \le \rho'(t) \le \mu(1-\widetilde{\sigma})\rho_0 e^{\mu(1- \widetilde{\sigma}) \widetilde{T}}.
\end{equation}
Taking $\widetilde{T}-\varepsilon$ (where $0<\varepsilon<\widetilde{T}$ is arbitrary) as a new initial time, we can extend the solution to $0\le t\le \widetilde{T}-\varepsilon +\delta(\widetilde{T})$ for some small $\delta(\widetilde{T})>0$ as in the previous proof. 
Using \re{wT1} and \re{wT2}, we see that $\delta(\widetilde{T})$ is independent of $\varepsilon$. If we take $\varepsilon<\delta(\widetilde{T})$, then we have
$$\widetilde{T}-\varepsilon + \delta(\widetilde{T}) > \widetilde{T},$$
which contradicts the assumption that $[0,\widetilde{T})$ is the maximum interval. Therefore, the solution exists for all $T>0$.}

\end{proof}

\section{Existence of an optimal control}

This section is devoted to the existence of a solution for the minimization problem \re{obj}.

\begin{thm}\label{exi-oc}
    Suppose that the assumptions in Theorem \ref{exi} are satisfied. Then, for any $T>0$, there exists an optimal control in $U_M$ that minimizes the objective functional $J(m)$ subject to \re{eqn1} -- \re{eqn6}, i.e., there is $m^*\in U_M$ such that $J(m^*)=\inf \{J(m)| m\in U_M\}$.
\end{thm}

\begin{proof}
    For a fixed $T>0$, from \re{est:u} and \re{est:rho} and the definition of $U_M$ in \re{set}, we have
    $$
    0\le J(m) \le \big(\rho_0 e^{\mu(1-\widetilde{\sigma})}T + BM^2\big) T. 
    $$
     So, $J(m)$ is uniformly bounded for any $m\in U_M$. Hence, there exists a minimizing sequence $\{m_n\}\subset U_M$ such that
    $$\lim\limits_{n\rightarrow \infty} J(m_n) = \inf\{J(m)\,|\,m\in U_M\}.$$
    Since $|m_n|\le M$ for all $n$, there exists a subsequence
    $$
    m^n\rightharpoonup m^* \quad \text{ weakly in $L^2(0,T)$}.
    $$
    By Theorem \ref{exi}, we define $(u^n,\rho^n)=(u(m^n),\rho(m^n))$ be the unique solution to the system \re{eqn1} -- \re{eqn6} corresponding to $m^n$. As in the proof of Theorem \ref{exi}, we have
    \begin{eqnarray}
\label{bound-u}        &\|u^n\|_{W^{2,1,p}((0,1)\times(0,T))}\le C,\\
\label{bound-rho}        &\|\rho^n\|_{C^1[0,T]}\le C,
    \end{eqnarray}
    where the constants $C$ are independent of $n$. Since the sequence $\{u^n\}$ is bounded in $W^{2,1,p}((0,1)\times(0,T))$, there exists $u^*\in W^{2,1,p}((0,1)\times(0,T))$ and a subsequence of $\{u^n\}$ that converges weakly to $u^*$ in $W^{2,1,p}((0,1)\times(0,T))$. Since the embedding $W^{2,1,p}((0,1)\times(0,T))\hookrightarrow C^{1+\alpha,(1+\alpha)/2}([0,1]\times[0,T])$ is compact, we have the following convergence on subsequences:
    \begin{eqnarray*}
        &u^n \rightarrow u^* \text{ in $C^{1,\frac12}([0,1]\times[0,T])$}, \quad  u^n_{\xi\xi}\rightharpoonup u^*_{\xi\xi}, u^n_t\rightharpoonup u^*_t  \text{ weakly in $L^2((0,1)\times(0,T))$}.
    \end{eqnarray*}
    For the sequence $\{\rho^n\}$, by \re{bound-rho}, and the Arzel\`a-Ascoli Theorem, we obtain on a subsequence
    $$\rho^n \rightarrow \rho^* \text{ uniformly in $C[0,T]$}.$$
    Furthermore, since $\|(\rho^n)'\|_{L^\infty(0,T)}\le \|\rho^n\|_{C^1[0,T]}\le C$ for all $n$, we also have
    $$\quad (\rho^n)' \rightharpoonup (\rho^*)' \text{ weakly in $L^2(0,T)$}.$$
    These convergences and uniform $L^\infty$ bounds imply that
    \begin{eqnarray*}
        \frac{(\rho^n)'}{\rho^n}\xi u^n_\xi \rightharpoonup \frac{(\rho^*)'}{\rho^*}\xi u^*_\xi,\; \frac{1}{(\rho^n)^2}u^n_{\xi\xi}\rightharpoonup \frac{1}{(\rho^*)^2}u^*_{\xi\xi},\; m^n u^n\rightharpoonup m^* u^* \quad \text{ weakly in $L^2((0,1)\times(0,T))$}.
    \end{eqnarray*}
Passing to the limit in the system \re{eqn1} -- \re{eqn6}, we see that $(u^*,\rho^*)$ is the unique solution associated with $m^*$. 

Finally, we need to verify that $m^*$ is an optimal control, i.e., $J(m^*)\le \inf\{J(m)|m\in U_M\}$. By the lower semicontinuity of the $L^2$ norm with respect to weak convergence, we have
$$\lim\limits_{n\rightarrow \infty} \inf \int_0^T \big(m^n(t)\big)^2 \dif t \ge \int_0^T \big(m^*(t)\big)^2 \dif t.$$
Thus,
\begin{equation*}
    \begin{split}
         \inf\{J(m)|m\in U_M\} = \lim\limits_{n\rightarrow \infty} J(m_n) &= \lim\limits_{n\rightarrow \infty} \int_0^T \Big(\rho^n(t) + B \big(m^n(t)\big)^2\Big)\dif t \\
         &\ge \int_0^T \Big(\rho^*(t) + B \big(m^*(t)\big)^2\Big) \dif t = J(m^*),
    \end{split}
\end{equation*}
which verifies that $m^*$ is an optimal control that minimizes the objective functional $J(m)$.
\end{proof}

\section{Derivation of the optimality system}

In this section, we characterize some properties of the optimal control by deriving the optimality system and necessary conditions for a minimizer. The optimality system consists of the state system, the adjoint system, and the characterization of the optimal control. Following the fundamental work of J. L. Lions \cite{lions1971optimal}, we first differentiate the solution $(u,\rho)$ to the state system with respect to the control $m$, which is guaranteed by the following theorem.

\begin{thm}\label{senthm}
Under the assumptions of Theorem \ref{exi}, let $m\in U_M$ and $(u,\rho)=(u^m,\rho^m)$ be the corresponding solution to the state system \re{eqn1} -- \re{eqn6}. Let $m_\varepsilon = m+\varepsilon h$ for any $m\in U_M$ and $h\in L^\infty(0,T)$ such that $m+\varepsilon h \in U_M$. For $T$ sufficiently small, the mapping $m\rightarrow (u^m,\rho^m)$ is differentiable in the following sense: there exists $(v,\eta) \in W^{2,1,p}((0,1)\times(0,T)) \times C^1[0,T]$ such that
\begin{eqnarray*}
        &\displaystyle v^\varepsilon \rightarrow v \text{ in $C^{1,\frac12}([0,1]\times[0,T])$}, \quad v^\varepsilon_{\xi\xi}\rightharpoonup v_{\xi\xi}, v^\varepsilon_t\rightharpoonup v_t  \text{ weakly in $L^2((0,1)\times(0,T))$},
    \end{eqnarray*}
$$\eta^\varepsilon \rightarrow \eta \text{ uniformly in $C[0,T]$ and $(\eta^\varepsilon)'\rightharpoonup \eta'$ weakly in $L^2(0,T)$},$$
as $\varepsilon \rightarrow 0$,
where $v^\varepsilon = \frac{u^{m+\varepsilon h}-u^m}{\varepsilon}$ and $\eta^\varepsilon = \frac{\rho^{m+\varepsilon h}-\rho^m}{\varepsilon}$.
Moreover, the sensitivity functions $(v,\eta)$ satisfy:
    \begin{equation}\label{sys2}
    \left\{
    \begin{aligned}
        &v_t - \frac{\rho'}{\rho}\xi v_{\xi} - \frac{1}{\rho^2} v_{\xi\xi}  +(1+m)v = \frac{\eta'}{\rho}\xi u_\xi -\frac{\rho'\eta}{\rho^2}\xi u_\xi - \frac{2\eta}{\rho^3}u_{\xi\xi} -u h \hspace{1em} &0<\xi<1,\;0<t<T,\\
        & v = 0 \hspace{1em} &\xi = 1, \; 0<t<T,\\
        & v_\xi = 0 \hspace{1em} &\xi=0,\;0<t<T,\\
        & v = 0 \hspace{1em} &0<\xi<1,\; t=0,\\
        & \eta'=\eta \int_0^1 \mu (u-\widetilde{\sigma})  \dif \xi  + \rho \int_0^1 \mu v\dif \xi \hspace{1em} & 0<t<T,\\
        & \eta = 0 \hspace{1em} & t=0.
    \end{aligned}
    \right.
\end{equation}
\end{thm}
\begin{proof}
    Denote $(u^\varepsilon, \rho^\varepsilon) = (u^{m+\varepsilon h},\rho^{m+\varepsilon h})$, which satisfies 
    \begin{equation}\label{sys-epi}
    \left\{
    \begin{aligned}
        &u^\varepsilon_t - \frac{(\rho^\varepsilon)'}{\rho^\varepsilon}\xi u^\varepsilon_{\xi} - \frac{1}{(\rho^\varepsilon)^2} u^\varepsilon_{\xi\xi}  =-(1+m+\varepsilon h)u^\varepsilon  \hspace{2em} &0<\xi<1,\;0<t<T,\\
        & u^\varepsilon = 1 \hspace{2em} &\xi = 1, \; 0<t<T,\\
        & u^\varepsilon_\xi = 0 \hspace{2em} &\xi=0,\; 0<t<T,\\
        & u^\varepsilon = u_0 \hspace{2em} &0<\xi<1,\; t=0,\\
        & (\rho^\varepsilon)'=\rho^\varepsilon \int_0^1 \mu(u^\varepsilon-\widetilde{\sigma}) \dif \xi \hspace{2em} & 0<t<T,\\
        &\rho^\varepsilon = \rho_0 \hspace{2em} & t=0.
    \end{aligned}
    \right.
\end{equation}
Let $(u,\rho)=(u^m,\rho^m)$, where
    \begin{equation}\label{sys-u}
    \left\{
    \begin{aligned}
        &u_t - \frac{\rho'}{\rho}\xi u_{\xi} - \frac{1}{\rho^2} u_{\xi\xi}  =-(1+m)u  \hspace{2em} &0<\xi<1,\;0<t<T,\\
        & u = 1 \hspace{2em} &\xi = 1, \; 0<t<T,\\
        & u_\xi = 0 \hspace{2em} &\xi=0,\; 0<t<T,\\
        & u = u_0 \hspace{2em} &0<\xi<1,\; t=0,\\
        & \rho'=\rho \int_0^1 \mu(u-\widetilde{\sigma}) \dif \xi \hspace{2em} & 0<t<T,\\
        &\rho = \rho_0 \hspace{2em} & t=0.
    \end{aligned}
    \right.
\end{equation}
Note that it follows from Theorem \ref{exi} that there is a constant $C>0$ independent of $\varepsilon$ such that
\begin{eqnarray}
\label{sen-est}
    &\displaystyle\|u^\varepsilon\|_{W^{2,1,p}((0,1)\times(0,T))} + \|\rho^\varepsilon\|_{C^1[0,T]}\le C,\\
\label{sen-u}
    &\displaystyle\|u\|_{W^{2,1,p}((0,1)\times(0,T))} + \|\rho\|_{C^1[0,T]}\le C,\\
    &\displaystyle 0<u^\varepsilon,u<1,\hspace{2em} 0< \rho_0 e^{-\mu\widetilde{\sigma} T} \le  \rho^\varepsilon,\rho \le \rho_0 e^{\mu(1-\widetilde{\sigma})T}.
\end{eqnarray}
Reasoning as in the proof of Theorem \ref{exi-oc} and by the uniqueness of the solution $(u,\rho)=(u^m,\rho^m)$, we have, as $\varepsilon\rightarrow 0$,
\begin{eqnarray}
        \label{uepi}\displaystyle &u^\varepsilon \rightarrow u \text{ in $C^{1,\frac12}([0,1]\times[0,T])$}, \quad u^\varepsilon _{\xi\xi}\rightharpoonup u_{\xi\xi}, u^\varepsilon _t\rightharpoonup u_t  \text{ weakly in $L^2((0,1)\times(0,T))$},\\
        \label{repi}\displaystyle &\rho^\varepsilon  \rightarrow \rho \text{ uniformly in $C[0,T]$}, \quad (\rho^\varepsilon)' \rightharpoonup \rho'  \text{ weakly in $L^2(0,T)$}.
    \end{eqnarray}

Recall that 
\begin{equation*}
    v^\varepsilon = \frac{u^{m+\varepsilon h}-u^m}{\varepsilon}=\frac{u^\varepsilon-u}{\varepsilon},\hspace{2em}\eta^\varepsilon = \frac{\rho^{m+\varepsilon h}-\rho^m}{\varepsilon}=\frac{\rho^\varepsilon-\rho}{\varepsilon},
\end{equation*}
and it is noteworthy that
\begin{eqnarray}
    &&\label{sen1} \frac{1}{\varepsilon}\Big(\frac{(\rho^\varepsilon)'}{\rho^\varepsilon} - \frac{\rho'}{\rho}\Big) = \frac{1}{\varepsilon}\frac{(\rho^\varepsilon)'\rho-\rho'\rho^\varepsilon}{\rho\rho^\varepsilon}=\Big(\frac{\rho^\varepsilon-\rho}{\varepsilon}\Big)'\frac{1}{\rho^\varepsilon}-\frac{\rho'}{\rho\rho^\varepsilon}\Big(\frac{\rho^\varepsilon-\rho}{\varepsilon}\Big)=\frac{(\eta^\varepsilon)'}{\rho^\varepsilon} - \frac{\rho'\eta^\varepsilon}{\rho\rho^\varepsilon},\\
    &&\label{sen2}\frac{1}{\varepsilon}\Big(\frac{1}{(\rho^\varepsilon)^2}-\frac{1}{\rho^2}\Big)= \frac{1}{\varepsilon}\frac{\rho^2-(\rho^\varepsilon)^2}{\rho^2(\rho^\varepsilon)^2} = -\frac{\rho^\varepsilon+\rho}{\rho^2(\rho^\varepsilon)^2}\Big(\frac{\rho^\varepsilon-\rho}{\varepsilon}\Big) = -\frac{(\rho^\varepsilon+\rho) \eta^\varepsilon}{\rho^2(\rho^\varepsilon)^2},\\
    &&\label{sen3}\frac{1}{\varepsilon}\int_0^1 \mu(u^\varepsilon-u)\dif \xi = \int_0^1 \mu\Big(\frac{u^\varepsilon-u}{\varepsilon}\Big) \dif \xi = \int_0^1 \mu v^\varepsilon \dif \xi.
\end{eqnarray}
Using \re{sen1} -- \re{sen3}, subtracting \re{sys-u} from \re{sys-epi} and dividing by $\varepsilon$, we obtain the following system:
    \begin{equation}\label{sen}
    \left\{
    \begin{aligned}
        &
        \begin{aligned}
&v^\varepsilon_t - \frac{(\rho^\varepsilon)'}{\rho^\varepsilon}\xi v^\varepsilon_{\xi} - \frac{1}{(\rho^\varepsilon)^2} v^\varepsilon_{\xi\xi}  +(1+m)v^\varepsilon = \\ &\hspace{6em} \frac{(\eta^\varepsilon)'}{\rho^\varepsilon}\xi u_\xi -\frac{\rho'\eta^\varepsilon}{\rho\rho^\varepsilon}\xi u_\xi - \frac{(\rho+\rho^\varepsilon)\eta^\varepsilon}{\rho^2(\rho^\varepsilon)^2}u_{\xi\xi} -u^\varepsilon h \end{aligned}\hspace{1em} &0<\xi<1,\;0<t<T,\\
        & v^\varepsilon = 0 \hspace{1em} &\xi = 1, \; 0<t<T,\\
        & v^\varepsilon_\xi = 0 \hspace{1em} &\xi=0,\;0<t<T,\\
        & v^\varepsilon = 0 \hspace{1em} &0<\xi<1,\; t=0,\\
        & (\eta^\varepsilon)'=\eta^\varepsilon  \int_0^1 \mu (u^\varepsilon-\widetilde{\sigma})  \dif \xi  + \rho \int_0^1 \mu v^\varepsilon \dif \xi \hspace{1em} & 0<t<T,\\
        & \eta^\varepsilon = 0 \hspace{1em} & t=0.
    \end{aligned}
    \right.
\end{equation}
By \re{sen-est}, the coefficients on the left-hand side of the PDE for $v^\epsilon$ are bounded in $L^\infty((0, 1)\times (0, T))$. Therefore, by artificially extending the time interval to $[0,1]$ as mentioned in the proof of Theorem \ref{exi}, applying the standard $L^p$ estimates for parabolic equations, and utilizing the embedding theorem, we can find positive constants $C_3(T)$ and $C_4(T)$ such that
\begin{equation}
    \label{v-est}
    \begin{split}
    \|v^\varepsilon\|_{C^{1+\alpha,(1+\alpha)/2}([0,1]\times[0,T])} &\le C\|v^\varepsilon\|_{W^{2,1,p}((0,1)\times(0,T))}\\
    &\le C(T)\Big\|\frac{(\eta^\varepsilon)'}{\rho^\varepsilon}\xi u_\xi -\frac{\rho'\eta^\varepsilon}{\rho\rho^\varepsilon}\xi u_\xi - \frac{(\rho+\rho^\varepsilon)\eta^\varepsilon}{\rho^2(\rho^\varepsilon)^2}u_{\xi\xi} -u^\varepsilon h\Big\|_{L^p((0,1)\times(0,T))}\\
    &\le C_3(T) \|\eta^\varepsilon\|_{C^1[0,T]} + C_4(T) \|h\|_{L^\infty(0,T)}.
    \end{split}
\end{equation}
Here and hereafter, $C_i(T)$ represent positive constants that are independent of $\varepsilon$ and depend on $T$; these constants are bounded for $T$ within any bounded set.
We then multiply the differential equation for $\eta^\varepsilon$ by $\eta^\varepsilon$, from which we derive 
\begin{equation*}
\begin{split}
    \frac12 \frac{\dif }{\dif t}(\eta^\varepsilon)^2 & = (\eta^\varepsilon)^2 \int_0^1 \mu(u^\varepsilon-\widetilde{\sigma}) \dif \xi + \rho\eta^\varepsilon \int_0^1 \mu v^\varepsilon\dif \xi\\
    &\le \frac{\mu(1-\widetilde{\sigma})}{3}(\eta^\varepsilon)^2 + \frac12\bigg((\eta^\varepsilon)^2 + \mu^2\|\rho\|_{L^\infty(0,T)}^2\|v^\varepsilon\|^2_{L^\infty((0,1)\times(0,T))}\bigg)\\
    &\le \frac{C_5}2 (\eta^\varepsilon)^2 + \frac{C_6(T)}2 \|v^\varepsilon\|_{L^\infty((0,1)\times(0,T))}^2.
\end{split}
\end{equation*}
By applying the Gronwall inequality, we deduce that
\begin{equation}\label{rho-est}
\begin{split}
    |\eta^\varepsilon(t)| &\le \sqrt{C_6(T) e^{C_5 T}} T^{1/2} \|v^\varepsilon\|_{L^\infty((0,1)\times(0,T))}
    \hspace{2em}\text{for $0\le t\le T$}.
    \end{split}
\end{equation}
As in the proof of Theorem \ref{exi}, we can use the equation for $\eta^\varepsilon$ and combine it with \re{rho-est} to derive the estimate for $(\eta^\varepsilon)'$, namely, for $0\le t\le T$,
\begin{equation}
    \label{rhod-est}
    \begin{split}
    |(\eta^\varepsilon)'(t)| &\le \mu(1-\widetilde{\sigma})\|\eta^\varepsilon\|_{L^\infty(0,T)} + \|\rho\|_{L^\infty(0,T)}\mu \|v^\varepsilon\|_{L^\infty((0,1)\times(0,T))} \\
    &\le \big(\mu(1-\widetilde{\sigma})\sqrt{C_6(T) e^{C_5 T}}T^{1/2}+\|\rho\|_{L^\infty(0,T)}\mu\big)\|v^\varepsilon\|_{L^\infty((0,1)\times(0,T))}.
    \end{split}
\end{equation}
Therefore, there exists positive constants $C_7(T)$ and $C_8(T)$ such that
\begin{equation*}
\begin{split}
    \|\eta^\varepsilon\|_{C^1[0,T]} = \|\eta^\varepsilon\|_{L^\infty(0,T)} + \|(\eta^\varepsilon)'\|_{L^\infty(0,T)} &\le C_7(T) \|v^\varepsilon\|_{L^\infty((0,1)\times(0,T))}\\
    &\le C_8(T) T^{1/2} \|v^\varepsilon\|_{C^{1+\alpha,(1+\alpha)/2}([0,1]\times[0,T])}.
    \end{split}
\end{equation*}
Substituting it into \re{v-est}, we have
\begin{equation}
    \|v^\varepsilon\|_{C^{1+\alpha,(1+\alpha)/2}([0,1]\times[0,T])} \le C_3(T) C_8(T) T^{1/2}   \|v^\varepsilon\|_{C^{1+\alpha,(1+\alpha)/2}([0,1]\times[0,T])} + C_4(T) \|h\|_{L^\infty(0,T)}.
\end{equation}
If $T>0$ is sufficiently small such that $C_3(T) C_8(T) T^{1/2} <1$, then 
\begin{eqnarray}
    \label{v-est1}& \displaystyle\ \|v^\varepsilon\|_{C^{1+\alpha,(1+\alpha)/2}([0,1]\times[0,T])}\le \frac{C_4(T)}{1-C_3(T) C_8(T) T^{1/2}} \|h\|_{L^\infty(0,T)} \le C,\\
    \label{rho-est1}&\displaystyle\ \|\eta^\varepsilon\|_{C^1[0,T]} \le \frac{C_4(T) C_8(T) T^{1/2}}{1-C_3(T) C_8(T) T^{1/2}}\|h\|_{L^\infty(0,T)} \le C.
\end{eqnarray}
The estimates \re{v-est1} and \re{rho-est1} justify the uniform boundedness of $v^\varepsilon$ and $\eta^\varepsilon$. Therefore, as $\varepsilon\rightarrow 0$, there exist $v\in W^{2,1,p}((0,1)\times(0,T))$ and $\eta \in W^{1,p}(0,T)$ such that
$$
v^\varepsilon \rightharpoonup v\; \text{ weakly in $W^{2,1,p}((0,1)\times(0,T))$},
$$
$$\eta^\varepsilon \rightarrow \eta \text{ uniformly in $C[0,T]$}, \quad (\eta^\varepsilon)' \rightharpoonup \eta'  \text{ weakly in $L^p(0,T)$}$$
Similar to the proof of Theorem \ref{exi-oc}, we have
\begin{eqnarray*}
        &\displaystyle v^\varepsilon \rightarrow v \text{ in $C^{1,\frac12}([0,1]\times[0,T])$}, \quad v^\varepsilon_{\xi\xi}\rightharpoonup v_{\xi\xi}, v^\varepsilon_t\rightharpoonup v_t  \text{ weakly in $L^2((0,1)\times(0,T))$}.
    \end{eqnarray*}
Furthermore, combining with \re{uepi} and \re{repi}, we also establish
\begin{eqnarray*}
        &\displaystyle\frac{(\rho^\varepsilon)'}{\rho^\varepsilon}\xi v^\varepsilon_\xi \rightharpoonup \frac{\rho'}{\rho}\xi v_\xi,\; \frac{1}{(\rho^\varepsilon)^2}v^\varepsilon_{\xi\xi}\rightharpoonup \frac{1}{\rho^2}v_{\xi\xi},\; \frac{(\eta^\varepsilon)'}{\rho^\varepsilon}\xi u_\xi \rightharpoonup \frac{\eta'}{\rho}\xi u_\xi \quad \text{ weakly in $L^2((0,1)\times(0,T))$},\\
        &\displaystyle (1+m)v^\varepsilon \rightarrow (1+m)v,\; \frac{\rho'\eta^\varepsilon}{\rho \rho^\varepsilon}\xi u_\xi \rightarrow \frac{\rho'\eta}{\rho^2}\xi u_\xi,\; \frac{(\rho+\rho^\varepsilon)\eta^\varepsilon}{\rho^2(\rho^\varepsilon)^2}u_{\xi\xi}\rightarrow \frac{2 \eta}{\rho^3}u_{\xi\xi}\quad  \text{ in $L^2((0,1)\times(0,T))$}.
    \end{eqnarray*}
Therefore, taking the limit as $\varepsilon \to 0$ in \re{sen} justifies that the pair $(v,\rho)$ solves the system \re{sys2}. 
Additionally, it follows from the equation of $\eta$ in the system \re{sys2} that $\eta\in C^1[0,T]$.
\end{proof}

To derive the optimality system and characterize the optimal control, it is necessary to define adjoint variables and the adjoint operator associated with the $(v,\eta)$ system. With this in mind, we rewrite the system \re{sys2} as 
\begin{equation}\label{sys-matrix}
    \mathcal{L}\begin{pmatrix}
        v \\ \eta
    \end{pmatrix}  = \begin{pmatrix}
        -uh \\ 0
    \end{pmatrix},
\end{equation}
where 
{\color{black}\begin{equation}
    \label{adj:L1}
    \mathcal{L}\begin{pmatrix}
        v \\ \eta
    \end{pmatrix} = \begin{pmatrix}
        L_1 v \\ L_2 \eta 
    \end{pmatrix} + 
        \mathcal{B} \begin{pmatrix}
        v \\ \eta
    \end{pmatrix},
\end{equation}
and $\mathcal{B}$ is the operator consisting of cross terms.
 From \re{sys2}, we have
\begin{equation}
    \label{adj:L2}
    \begin{pmatrix}
        L_1 v \\ L_2 \eta 
    \end{pmatrix} = \begin{pmatrix}
        v_t - \frac{\rho'}{\rho}\xi v_{\xi} - \frac{1}{\rho^2} v_{\xi\xi}  +(1+m)v \\
        \eta' - \eta \int_0^1 \mu (u-\widetilde{\sigma})  \dif \xi  
    \end{pmatrix},
\end{equation}
and
\begin{equation}
\begin{split}
    \label{adj:L3}
    \mathcal{B}\begin{pmatrix}
        v \\ \eta
    \end{pmatrix} &= \begin{pmatrix}
        - \frac{\eta'}{\rho}\xi u_\xi +\frac{\rho'\eta}{\rho^2}\xi u_\xi + \frac{2\eta}{\rho^3}u_{\xi\xi} \\
        - \rho \int_0^1 \mu v \dif \xi 
    \end{pmatrix} \\
    &= \begin{pmatrix}
        -\frac{\eta}{\rho}\big(\int_0^1 \mu(u-\widetilde{\sigma})\dif \xi\big)\xi u_\xi -\big(\int_0^1\mu v\dif \xi\big)\xi u_\xi + \frac{\eta}{\rho}\big(\int_0^1 \mu(u-\widetilde{\sigma})\dif\xi\big)\xi u_\xi + \frac{2\eta}{\rho^3}u_{\xi\xi}\\
        - \rho \int_0^1 \mu v \dif \xi
    \end{pmatrix}\\
    &= \begin{pmatrix}
         -\big(\int_0^1\mu v\dif \xi\big)\xi u_\xi  + \frac{2\eta}{\rho^3}u_{\xi\xi}\\
        - \rho \int_0^1 \mu v \dif \xi
    \end{pmatrix}.
    \end{split}
\end{equation}
Note that in defining the operator $\mathcal{B}$, we have substituted $\eta'$ with its equation from the system \re{sys2} and $\rho'$ with its expression from the equation \re{eqn4} to eliminate the derivatives.}

For $f_1,f_2 \in L^2((0,1)\times(0,T))$ and $g_1, g_2 \in L^2(0,T)$, we define the following inner product
\begin{equation}
    \label{innerp}
    \langle \begin{pmatrix}
        f_1 \\ g_1
    \end{pmatrix}, \begin{pmatrix}
        f_2 \\ g_2
    \end{pmatrix}\rangle = \int_0^T \int_0^1 f_1 f_2 \dif \xi \dif t + \int_0^T g_1 g_2 \dif t.
\end{equation}

Let $(w,\lambda)$ be the adjoint variables corresponding to $(v,\eta)$. The adjoint PDE system is defined as
\begin{equation}
\displaystyle
    \label{adj:L4}
    \mathcal{L}^*\begin{pmatrix}
        w \\ \lambda
    \end{pmatrix} = \begin{pmatrix}
        \frac{\partial (\text{integrand of $J$})}{\partial \sigma} \\ \frac{\partial (\text{integrand of $J$})}{\partial \rho}
    \end{pmatrix} = \begin{pmatrix}
        0 \\ 1
    \end{pmatrix},
\end{equation}
where 
\begin{equation}
    \label{adj:L5}
    \mathcal{L}^*\begin{pmatrix}
        w \\ \lambda
    \end{pmatrix} = \begin{pmatrix}
        L_1^* w \\ L_2^* \lambda
    \end{pmatrix} + \mathcal{B}^*\begin{pmatrix}
        w \\ \lambda
    \end{pmatrix},
\end{equation}
and we require
\begin{equation}
    \label{adj:L6}
    \langle \mathcal{L}\begin{pmatrix}
        v \\ \eta
    \end{pmatrix}, \begin{pmatrix}
        w \\ \lambda 
    \end{pmatrix}\rangle = \langle \begin{pmatrix}
        v \\ \eta 
    \end{pmatrix}, \mathcal{L}^* \begin{pmatrix}
        w \\ \lambda
    \end{pmatrix}\rangle.
\end{equation}
Since the inner product \re{innerp} is linear, the requirement \re{adj:L6} is equivalent to
\begin{eqnarray}
    &&\displaystyle \label{req1} \langle \begin{pmatrix}
        L_1 v \\ L_2 \eta
    \end{pmatrix}, \begin{pmatrix}
        w \\ \lambda 
    \end{pmatrix}\rangle = \langle \begin{pmatrix}
        v \\ \eta 
    \end{pmatrix}, \begin{pmatrix}
        L_1^* w \\ L_2^* \lambda
    \end{pmatrix}\rangle,\\
    && \label{req2} \displaystyle \langle \mathcal{B}\begin{pmatrix}
        v \\ \eta
    \end{pmatrix}, \begin{pmatrix}
        w \\ \lambda 
    \end{pmatrix}\rangle = \langle \begin{pmatrix}
        v \\ \eta 
    \end{pmatrix}, \mathcal{B}^* \begin{pmatrix}
        w \\ \lambda
    \end{pmatrix}\rangle.
\end{eqnarray}
In the following, we shall use \re{req1} and \re{req2} to find $L_1^*$, $L_2^*$, and $\mathcal{B}^*$. Using integration by parts, it is easy to verify
\begin{equation}\label{operatorL}
\begin{split}
&\langle \begin{pmatrix}
        L_1 v \\ L_2 \eta
    \end{pmatrix}, \begin{pmatrix}
        w \\ \lambda 
    \end{pmatrix}\rangle \\
    &\hspace{1em}= \int_0^T \int_0^1 w\Big(v_t - \frac{\rho'}{\rho}\xi v_{\xi} - \frac{1}{\rho^2} v_{\xi\xi}  +(1+m)v\Big)\dif \xi \dif t + \int_0^T \lambda\Big(\eta' - \eta \int_0^1 \mu (u-\widetilde{\sigma})\dif \xi\Big)\dif t\\
    &\hspace{1em}=\int_0^T\int_0^1 v\Big(-w_t + \frac{\rho'}{\rho}(\xi w)_\xi -\frac{1}{\rho^2}w_{\xi\xi} + (1+m)w\Big)\dif \xi \dif t + \int_0^T \eta\Big(-\lambda' - \lambda \int_0^1 \mu(u-\widetilde{\sigma})\dif \xi \Big)\dif t,
    \end{split}
\end{equation}
if the following boundary and terminal conditions are provided:
\begin{eqnarray*}
    \displaystyle&&
    w(1,t)=0\quad  \text{ and }\quad w_\xi(0,t)=0 \hspace{2em}0<t<T,\\
    \displaystyle&&
    w(x,T) = 0\quad \text{ and } \quad \lambda(T)=0 \hspace{2em}0<x<1.
\end{eqnarray*}
Therefore, we obtain
\begin{equation}
    \label{adj:L7}
    \begin{pmatrix}
        L_1^* w \\ L_2^* \lambda
    \end{pmatrix} = \begin{pmatrix}
        -w_t + \frac{\rho'}{\rho}(\xi w)_\xi -\frac{1}{\rho^2}w_{\xi\xi} + (1+m)w \\
        -\lambda' - \lambda \int_0^1 \mu(u-\widetilde{\sigma})\dif \xi
    \end{pmatrix}.
\end{equation}
In deriving $\mathcal{B}^*$ from \re{adj:L3} and \re{req2}, we rearrange the terms and use integration by parts on the integral $\int_0^1 w u_{\xi\xi}\dif \xi$ to obtain
\begin{equation}\label{operatorB}
\begin{split}
    &\langle \mathcal{B}\begin{pmatrix}
        v \\ \eta
    \end{pmatrix},\begin{pmatrix}
        w \\ \lambda
    \end{pmatrix}\rangle \\
    &\hspace{1em}= \int_0^T \int_0^1 w\Big(-\big(\int_0^1\mu v\dif \chi\big)\xi u_\xi  + \frac{2\eta}{\rho^3}u_{\xi\xi}\Big) \dif \xi \dif t + \int_0^T \lambda \Big(-\rho \int_0^1 \mu v \dif \xi \Big)\dif t\\
    &\hspace{1em}=\int_0^T  \big(\int_0^1 \mu v\dif \chi\big)\big(-\int_0^1w\xi u_\xi \dif \xi \big)\dif t + \int_0^T \frac{2\eta}{\rho^3}\int_0^1 w u_{\xi\xi}\dif \xi \dif t - \int_0^T\int_0^1 \lambda \rho \mu v\dif \xi \dif t\\
    &\hspace{1em}= \int_0^T \int_0^1 v \Big(-\mu \int_0^1 w \xi u_\xi \dif \xi \Big) \dif \chi \dif t\, -\, \int_0^T \frac{2\eta}{\rho^3} \int_0^1 w_\xi u_\xi \dif \xi \dif t \, - \,  \int_0^T \int_0^1 v\Big(\mu \lambda \rho \Big)\dif \xi \dif t\\
    &\hspace{1em}= \int_0^T \int_0^1 v\Big(-\mu \int_0^1 w \xi u_\xi \dif \xi  - \mu \lambda \rho \Big) \dif \chi \dif t \, + \, \int_0^T \eta \Big(-\frac{2} {\rho^3}\int_0^1 w_\xi u_\xi \dif \xi \Big)\dif t.
\end{split}
\end{equation}
Thus, $\mathcal{B}^*$ is defined as
\begin{equation}
    \label{adj:L8}
    \mathcal{B}^* \begin{pmatrix}
        w \\ \lambda 
    \end{pmatrix} = \begin{pmatrix}
        -\mu \int_0^1 w \xi u_\xi \dif \xi  - \mu \lambda \rho \\
        -\frac{2} {\rho^3}\int_0^1 w_\xi u_\xi \dif \xi
    \end{pmatrix}.
\end{equation}

By combining \re{adj:L4} with \re{adj:L5}, \re{adj:L7}, and \re{adj:L8}, we derive the adjoint PDE system. The existence of a solution for this system is guaranteed by the following theorem:

\begin{thm}\label{thm:adj}
Suppose the assumptions of Theorem \ref{exi} are satisfied. Let $m\in U_M$ and $(u,\rho)=(u^m,\rho^m)$ be the corresponding solution to the system \re{eqn1} -- \re{eqn6}. For sufficiently small $T>0$ and any $p>1$, there exist $w\in W^{2,1,p}((0,1)\times(0,T))$ and $\lambda \in C^1[0,T]$ such that
\begin{equation}\label{adj}
    \left\{
    \begin{aligned}
        & - w_t + \frac{\rho'}{\rho}(\xi w)_\xi -\frac{1}{\rho^2}w_{\xi \xi} + (1+m)w-\mu\int_0^1 w \xi u_\xi \dif \xi-\mu\lambda \rho =0 \hspace{1em} &0<\xi<1,\;0<t<T,\\
        & w = 0 \hspace{2em} &\xi = 1, \; 0<t<T,\\
        & w_\xi = 0 \hspace{2em} &\xi=0,\;0<t<T,\\
        & w = 0 \hspace{2em} &0<\xi<1,\; t=T,\\
        & - \lambda' - \lambda \int_0^1 \mu(u-\widetilde{\sigma})\dif \xi  - \frac{2}{\rho^3}\int_0^1 w_\xi u_\xi\dif \xi =1  \hspace{2em} & 0<t<T,\\
        & \lambda = 0 \hspace{2em}& t=T.
    \end{aligned}
    \right.
\end{equation}
\end{thm}
\begin{proof}
    We shall use the contraction mapping principle, similar to the proof of Theorem \ref{exi}, to show that the system \re{adj} admits a unique solution $(w,\lambda)$. 
    Consider the Banach space
    $$\mathcal{K}_2=\{(w,\lambda)\in L^\infty((0,1)\times(0,T))\times L^\infty(0,T); \, \|w\|_{L^\infty((0,1)\times(0,T))} + \|\lambda\|_{L^\infty(0,T)} \le K_2\}.$$
For any $(l,\psi)\in \mathcal{B}_1$, we define $(w,\lambda)$ by the solution of the following system:
\begin{equation}\label{adj-proof1}
    \left\{
    \begin{aligned}
    &\begin{aligned}
        &- w_t + \frac{\rho'}{\rho} (\xi w)_\xi -\frac{1}{\rho^2}w_{\xi\xi} + (1+m)w= \mu\int_0^1 l \xi u_\xi \dif \xi + \mu\psi \rho
    \end{aligned}
         \hspace{2em} &0<\xi<1,\;0<t<T,\\
        & w = 0 \hspace{2em} &\xi = 1, \; 0<t<T,\\
        & w_\xi = 0 \hspace{2em} &\xi=0,\;0<t<T,\\
        & w = 0 \hspace{2em} &0<\xi<1,\; t=T,\\
        &-\lambda' = 1 +  \psi \int_0^1 \mu(u-\widetilde{\sigma})\dif \xi  + \frac{2}{\rho^3}\int_0^1 w_\xi u_\xi\dif \xi  \hspace{2em} & 0<t<T,\\
        & \lambda = 0 \hspace{2em}& t=T.
    \end{aligned}
    \right.
\end{equation}
We start with analyzing the system for $w$. By Theorem \ref{exi}, there exists $C(T)>0$ such that
\begin{eqnarray*}
    &\|u\|_{W^{2,1,p}((0,1)\times(0,T))} + \|\rho\|_{C^1[0,T]}\le C(T),\\
    &0< u<1,\hspace{2em} 0< \rho_0 e^{-\mu\widetilde{\sigma} T} \le  \rho \le \rho_0 e^{\mu(1-\widetilde{\sigma})T}.
\end{eqnarray*}
Here and hereafter, $C(T)$ represents constants that are bounded for  $T$ in any bounded set. Note that the coefficients in the parabolic equation for $w$ are bounded. 
So, we can introduce a change of variable $t\rightarrow T-t$ and then apply the parabolic $L^p$ estimate
to obtain
\begin{equation}\label{thm5.2-1}
    \|w\|_{W^{2,1,p}((0,1)\times(0,T))} \le C(T)\Big\|\mu\int_0^1 l \xi u_\xi \dif \xi + \mu\psi \rho\Big\|_{L^p((0,1)\times(0,T))} \le C(T).
\end{equation}
Using the embedding theorem as in the proof of Theorem \ref{exi}, 
\begin{equation}\label{thm5.2-3}
    \|w\|_{C^{1+\alpha,(1+\alpha)/2}([0,1]\times[0,T])} \le C(T)\|w\|_{W^{2,1,p}((0,1)\times(0,T))} \le C(T).
\end{equation}
For any $x\in[0,1]$ and $t\in[0,T]$, noticing  $w(x,T)=0$, it follows that
\begin{equation*}
    \begin{split}
        |w(x,t)|=|w(x,T)-w(x,t)|&\le \sup\limits_{t_1,t_2\in [0,T]}\frac{|w(x,t_1)-w(x,t_2)|}{|t_1-t_2|^{1/2}}|t_1-t_2|^{1/2}\\
        &\le \|w\|_{C^{1+\alpha,(1+\alpha)/2}([0,1]\times[0,T])} T^{1/2}.
    \end{split}
\end{equation*}
Thus, we have
\begin{equation}\label{bound-lambda1}
    \|w\|_{L^\infty((0,1)\times(0,T))} \le T^{1/2} \|w\|_{C^{1+\alpha,(1+\alpha)/2}([0,1]\times[0,T])} \le C(T) T^{1/2}.
\end{equation}

We then analyze the equation for $\lambda$. By Theorem \ref{exi} and the estimate \re{thm5.2-3}, we know that $u,w\in C^{1+\alpha,(1+\alpha)/2}([0,1]\times[0,T])$. Hence, the right-hand side of the equation for $\lambda$ is continuous and bounded. Integrating the equation of $\lambda$ over the interval $[t,T]$ and recalling that $\lambda(T)=0$, we obtain
\begin{equation}\label{bound-lambda2}
    \|\lambda\|_{L^\infty(0,T)} \le C(T) T.
\end{equation}
Thus, if we define a mapping $\mathcal{M}_2$ by $$(w,\lambda)=\mathcal{M}_2(l,\psi),$$
then it is clear from \re{bound-lambda1} and \re{bound-lambda2} that $\mathcal{M}_2$ maps $\mathcal{K}_2$ into itself when $T>0$ is sufficiently small.

To further prove the mapping is a contraction, let $(l_i, \psi_i)\in\mathcal{B}_1$ and denote 
$$
(w_i,\lambda_i)=\mathcal{M}_2(l_i,\psi_i) \quad\text{ for  }   i=1,2.
$$ 
Since the equations in the system \re{adj-proof1} are linear, we have
\begin{equation*}
    \|w_1 - w_2\|_{W^{2,1,p}((0,1)\times(0,T))}  \le C(T) \big(\|l_1 - l_2\|_{L^\infty((0,1)\times(0,T))} + \|\psi_1 - \psi_2 \|_{L^\infty(0,T)} \big),
\end{equation*}
which, by the Sobolev embedding theorem, implies
\begin{equation*}
    \|w_1 - w_2\|_{C^{1+\alpha,(1+\alpha)/2}([0,1]\times[0,T])}  \le C(T) \big(\|l_1 - l_2\|_{L^\infty((0,1)\times(0,T))} + \|\psi_1 - \psi_2 \|_{L^\infty(0,T)} \big).
\end{equation*}
Hence,
\begin{equation}\label{adjp1}
    \|w_1 - w_2\|_{L^\infty ([0,1]\times[0,T])}  \le C(T) T^{1/2} \big(\|l_1 - l_2\|_{L^\infty((0,1)\times(0,T))} + \|\psi_1 - \psi_2 \|_{L^\infty(0,T)} \big).
\end{equation}
Subtracting the equations for $\lambda_1$ and $\lambda_2$ and integrating over $[0,T]$, we obtain
\begin{equation}\label{adjp2}
\begin{split}
    \|\lambda_1 - \lambda_2\|_{L^\infty(0,T)}  &\le C(T) T\big(\|\psi_1 - \psi_2 \|_{L^\infty(0,T)} + \|w_1 - w_2\|_{C^{1,1/2}([0,1]\times[0,T])} \big)\\
    &\le C(T) T\big(\|\psi_1 - \psi_2 \|_{L^\infty(0,T)} + CT^{\alpha/2} \|w_1 - w_2\|_{C^{1+\alpha,(1+\alpha)/2}([0,1]\times[0,T])} \big)\\
    & \le C(T) T \big(\|l_1 - l_2\|_{L^\infty((0,1)\times(0,T))} + \|\psi_1 - \psi_2 \|_{L^\infty(0,T)}\big).
    \end{split}
\end{equation}
By \re{adjp1} and \re{adjp2},  the mapping $\mathcal{M}_2$ is a contraction when $T>0$ is sufficiently small. Finally, we utilize the estimate \re{thm5.2-1} and the equation for $\lambda'$ to deduce that $w\in W^{2,1,p}((0,1)\times(0,T))$ and $\lambda \in C^1[0,T]$. The proof is complete.
\end{proof}

After deriving the sensitive system and the adjoint system, we are now ready to characterize the optimal control. Let $m$ be an optimal control, and $(u,\rho)$ the corresponding solution of the system \re{eqn1} -- \re{eqn6}. Suppose $m+\varepsilon h\in U_M$ for $\varepsilon >0$, and $(u^\varepsilon,\rho^\varepsilon)$ is the corresponding solution. We compute the directional derivative of the objective functional $J(m)$ with respect to $m$ in the direction of $h$. Since $J(m)$ is the minimum value,
\begin{equation*}
\begin{split}
    0&\le \lim\limits_{\varepsilon\rightarrow 0^+} \frac{J(m+\varepsilon h)-J(m)}{\varepsilon}\\
    &=\lim\limits_{\varepsilon\rightarrow 0^+} \int_0^T \frac{\rho^\varepsilon-\rho}{\varepsilon}\dif t + \lim\limits_{\varepsilon\rightarrow 0^+} \int_0^T B(2mh + \varepsilon h^2)\dif t\\
    &=\int_0^T \eta \dif t + \int_0^T 2Bmh \dif t \\
    &= \int_0^T \langle \begin{pmatrix}
        v \\ \eta
    \end{pmatrix}, \begin{pmatrix}
        0 \\ 1
    \end{pmatrix}\rangle \dif t + \int_0^T 2Bmh \dif t\\
    &=\int_0^T \langle \begin{pmatrix}
        v \\ \eta
    \end{pmatrix}, \mathcal{L}^* \begin{pmatrix}
        w \\ \lambda
    \end{pmatrix}\rangle \dif t + \int_0^T 2Bmh \dif t =\int_0^T \langle \mathcal{L}\begin{pmatrix}
         v \\  \eta
    \end{pmatrix}, \begin{pmatrix}
        w \\ \lambda
    \end{pmatrix}\rangle \dif t + \int_0^T 2Bmh \dif t\\
    &=\int_0^T \langle \begin{pmatrix}
         -uh \\  0
    \end{pmatrix}, \begin{pmatrix}
        w \\ \lambda
    \end{pmatrix}\rangle \dif t + \int_0^T 2Bmh \dif t= \int_0^T \int_0^1 -w u h \dif \xi \dif t + \int_0^T 2Bmh \dif t\\
    &=\int_0^T h\Big(2Bm - \int_0^1 w u \dif \xi \Big)\dif t,
\end{split}
\end{equation*}
where we used Theorem \ref{senthm} and equations \re{sys-matrix} and \re{adj:L4}. Therefore, we arrive at
\begin{equation*}
    \int_0^T h\Big(2Bm - \int_0^1 w u \dif \xi \Big)\dif t \ge 0,
\end{equation*}
for any $h\in L^\infty(0,T)$ such that $m+\varepsilon h\in U_M$. By a standard  argument, we obtain the characterization of the optimal control:
\begin{equation}
    \label{optimalc}
    m^*(t) = \min\Big\{\max\Big\{0,\frac{\int_0^1 w(\xi,t) u(\xi,t) \dif \xi}{2B} \Big\}, M\Big\},
\end{equation}
where $u$ is the solution to the state system, and $w$ is the corresponding adjoint variable.

Substituting \re{optimalc} into the state system \re{eqn1} -- \re{eqn6} and the adjoint system \re{adj}, we obtain the optimality system:
{\allowdisplaybreaks\begin{align}\label{os}
    &\\
        \nonumber &u_t - \frac{\rho'}{\rho}\xi u_{\xi} - \frac{1}{\rho^2} u_{\xi\xi} + (1+m^*)u =0\hspace{2em} \nonumber \nonumber &0<\xi<1,\;0<t<T,\\
        \nonumber & u = 1 \hspace{2em} &\xi = 1, \; 0<t<T,\\
        \nonumber & u_\xi = 0 \hspace{2em} &\xi=0,\; 0<t<T,\\
        \nonumber& u = u_0 \hspace{2em} &0<\xi<1,\; t=0,\\
        \nonumber& \rho'=\rho \int_0^1 \mu(u-\widetilde{\sigma}) \dif \xi \hspace{2em} & 0<t<T,\\
        \nonumber&\rho = \rho_0 \hspace{2em} & t=0,\\
        \nonumber&\begin{aligned}
        &-w_t + \frac{\rho'}{\rho}(\xi w)_\xi -\frac{1}{\rho^2}w_{\xi\xi} + (1+m^*)w = \mu\int_0^1 w \xi u_\xi \dif \xi + \mu\lambda \rho
    \end{aligned}
         &0<\xi<1,\;0<t<T,\\
        \nonumber& w = 0 \hspace{2em} &\xi = 1, \; 0<t<T,\\
        \nonumber& w_\xi = 0 \hspace{2em} &\xi=0,\;0<t<T,\\
        \nonumber& w = 0 \hspace{2em} &0<\xi<1,\; t=T,\\
        \nonumber& 
            -\lambda' = 1 +  \lambda \int_0^1 \mu(u-\widetilde{\sigma})\dif \xi  + \frac{2}{\rho^3}\int_0^1 w_\xi u_\xi\dif \xi
          \hspace{1em} & 0<t<T,\\
        \nonumber& \lambda = 0 \hspace{2em}& t=T,\\
        \nonumber&m^*(t) = \min\Big\{\max\Big\{0,\frac{\int_0^1 w(\xi,t) u(\xi,t) \dif \xi}{2B} \Big\}, M\Big\} \hspace{2em} & 0<t<T.
\end{align}}

The existence of a solution to the optimality system is guaranteed by Theorems \ref{exi} and \ref{thm:adj}. For a sufficiently small final time $T$, we now proceed to prove the uniqueness of this solution, thereby providing a characterization, as shown in \re{optimalc}, of the unique optimal control based on the solution of the optimality system.

\begin{thm}\label{thm:os}
    Suppose the assumptions of Theorem \ref{exi} are satisfied. For sufficiently small $T>0$, the optimality system \re{os} has a unique solution $(u,\rho,w,\lambda)$ such that $u,w\in W^{2,1,p}((0,1)\times(0,T))$ and $\rho,\lambda\in C^1[0,T]$.
\end{thm}
\begin{proof}
    Suppose that $(u,\rho,w,\lambda)$ and $(\bar{u},\bar{\rho},\bar{w},\bar{\lambda})$ are two solutions to the optimality system \re{os}. By Theorems \ref{exi} and \ref{thm:adj}, we have 
    $$\|u\|_{W^{2,1,p}((0,1)\times(0,T))} + \|w\|_{W^{2,1,p}((0,1)\times(0,T))} + \|\rho|_{C^1[0,T]} + \|\lambda\|_{C^1[0,T]} \le C(T),$$
    $$\|\bar{u}\|_{W^{2,1,p}((0,1)\times(0,T))} + \|\bar{w}\|_{W^{2,1,p}((0,1)\times(0,T))} + \|\bar{\rho}\|_{C^1[0,T]} + \|\bar{\lambda}\|_{C^1[0,T]} \le C(T),$$
where $C(T)$ and $C_i(T)$, hereafter, are bounded for $T$ in any bounded set.
    We denote 
    \begin{equation*}
        m^*=\min\Big\{\max\Big\{0,\frac{\int_0^1 w u \dif \xi}{2B} \Big\}, M\Big\},
    \end{equation*}
    and 
    \begin{equation*}
        \bar{m}^*=\min\Big\{\max\Big\{0,\frac{\int_0^1 \bar{w} \bar{u} \dif \xi}{2B} \Big\}, M\Big\},
    \end{equation*}
    respectively. It follows from the definition of $m^*$ and $\bar m^*$ that
\begin{equation*}
    \|m^*-\bar{m}^*\|_{L^\infty(0,T)} \le C\big(\|w-\bar{w}\|_{L^\infty((0,1)\times(0,T))} + \|u-\bar{u}\|_{L^\infty((0,1)\times(0,T))}\big).
\end{equation*}

Subtracting the equations for $u$ and $\bar{u}$, and employing similar techniques as in the proof of Theorem \ref{exi}, we obtain
\begin{equation*}
\begin{split}
     &\|u-\bar{u}\|_{C^{1+\alpha,(1+\alpha)/2}([0,1]\times[0,T])}\\
    &\hspace{2em}\le C(T)\big(\|\rho-\bar{\rho}\|_{C^1[0,T]} + \|m^*-\bar{m}^*\|_{L^\infty(0,T)}\big)\\
    &\hspace{2em}\le C(T)\big(\|\rho-\bar{\rho}\|_{C^1[0,T]} + \|w-\bar{w}\|_{L^\infty((0,1)\times(0,T))} + \|u-\bar{u}\|_{L^\infty((0,1)\times(0,T))}\big)\\
    &\hspace{2em}\le C(T)\big(\|\rho-\bar{\rho}\|_{C^1[0,T]} + \|w-\bar{w}\|_{C^{1,1/2}([0,1]\times[0,T])} + \|u-\bar{u}\|_{C^{1,1/2}([0,1]\times[0,T])}\big),
    \end{split}
\end{equation*}
Taking the initial conditions of $u$ and $\bar{u}$ into consideration yields
\begin{equation}\label{ub}
\begin{split}
     &\|u-\bar{u}\|_{C^{1,1/2}([0,1]\times[0,T])} \\
     &\hspace{2em}\le C(T) T^{\alpha/2} \|u-\bar{u}\|_{C^{1+\alpha,(1+\alpha)/2}([0,1]\times[0,T])}\\
     &\hspace{2em}\le C_9(T) T^{\alpha/2} \big(\|\rho-\bar{\rho}\|_{C^1[0,T]} + \|w-\bar{w}\|_{C^{1,1/2}([0,1]\times[0,T])} + \|u-\bar{u}\|_{C^{1,1/2}([0,1]\times[0,T])}\big).
     \end{split}
\end{equation}
Then, for the equations of $\rho$ and $\bar{\rho}$, we apply techniques similar to those used in the proofs of Theorems \ref{exi} and \ref{senthm} to obtain
\begin{equation*}
    \|\rho -\bar{\rho}\|_{L^\infty(0,T)} \le C(T)T^{1/2} \|u-\bar{u}\|_{L^\infty((0,1)\times(0,T))},
\end{equation*}
and
\begin{equation*}
    \|(\rho-\bar{\rho})'\|_{L^\infty(0,T)} \le C\|\rho-\bar{\rho}\|_{L^\infty(0,T)} + C \|u-\bar{u}\|_{L^\infty((0,1)\times(0,T))} \le C(T)\|u-\bar{u}\|_{L^\infty((0,1)\times(0,T))}.
\end{equation*}
Hence,
\begin{equation}
    \label{rb}
    \|\rho-\bar{\rho}\|_{C^1[0,T]}\le C(T)\|u-\bar{u}\|_{L^\infty((0,1)\times(0,T))} \le C_{10}(T) T^{1/2}\|u-\bar{u}\|_{C^{1,1/2}([0,1]\times[0,T])}.
\end{equation}

Next, we derive estimates for $w-\bar{w}$ and $\lambda-\bar{\lambda}$. By subtracting the governing equations of $w$ and $\bar{w}$ and applying the $L^p$ estimates for the resulting parabolic equation, we obtain
\begin{equation*}
\begin{split}
    &\|w-\bar{w}\|_{W^{2,1,p}((0,1)\times(0,T))} \\
    &\hspace{2em}\le C(T)\big(\|\rho-\bar{\rho}\|_{C^1[0,T]} + \|w-\bar{w}\|_{L^\infty(0,T)} + \|m^*-\bar{m}^*\|_{L^\infty(0,T)} +  \|u-\bar{u}\|_{C^{1,1/2}([0,1]\times[0,T])} \\
    &\hspace{4.5em}+ \|\lambda -\bar{\lambda}\|_{L^\infty(0,T)}\big)\\
    &\hspace{2em}\le C(T)\big(\|\rho-\bar{\rho}\|_{C^1[0,T]} + \|w-\bar{w}\|_{C^{1,1/2}([0,1]\times[0,T])} + \|\lambda -\bar{\lambda}\|_{L^\infty(0,T)} + \|u-\bar{u}\|_{C^{1,1/2}([0,1]\times[0,T])}\big).
    \end{split}
\end{equation*}
Using the embedding theorem, after artificially extending the time interval to $[0,1]$ as applied in the proof of Theorem \ref{exi} to avoid dependency of the embedding constant on $T$, we further have
\begin{equation*}
    \|w-\bar{w}\|_{C^{1,1/2}([0,1]\times[0,T])} \le C T^{\alpha/2} \|w-\bar{w}\|_{C^{1+\alpha,(1+\alpha)/2}([0,1]\times[0,T])} \le C T^{\alpha/2}\|w-\bar{w}\|_{W^{2,1,p}((0,1)\times(0,T))}.
\end{equation*}
This implies
\begin{equation}\label{wb}
\begin{split}
    &\|w-\bar{w}\|_{C^{1,1/2}([0,1]\times[0,T])} \le C_{11}(T) T^{\alpha/2} \big(\|\rho-\bar{\rho}\|_{C^1[0,T]} + \|w-\bar{w}\|_{C^{1,1/2}([0,1]\times[0,T])} \\
    &\hspace{16.7em}+ \|\lambda -\bar{\lambda}\|_{L^\infty(0,T)} + \|u-\bar{u}\|_{C^{1,1/2}([0,1]\times[0,T])}\big).
    \end{split}
\end{equation}
Then, subtracting the equations for $\lambda$ and $\bar{\lambda}$ and  integrating the resulting equation, we obtain
\begin{equation}\label{lb}
\begin{split}
    \|\lambda -\bar{\lambda}\|_{L^\infty(0,T)} &\le C_{12}(T) T \big( \|\lambda -\bar{\lambda}\|_{L^\infty(0,T)} + \|u-\bar{u}\|_{C^{1,1/2}([0,1]\times[0,T])}  \\
    &\hspace{5.3em}+ \|w-\bar{w}\|_{C^{1,1/2}([0,1]\times[0,T])}+ \|\rho-\bar{\rho}\|_{C^1[0,T]}\big).
    \end{split}
\end{equation}


Adding \re{ub}, \re{rb}, \re{wb}, and \re{lb}, we arrive at the inequality
\begin{equation*}
\begin{split}
    &\Big(1-C_9(T) T^{\alpha/2} - C_{10}(T)T^{1/2} - C_{11}(T)T^{\alpha/2} - C_{12}(T) T\Big)\|u-\bar{u}\|_{C^{1,1/2}([0,1]\times[0,T])} \\
    &\hspace{2em}+ \Big(1-C_9(T) T^{\alpha/2} - C_{11}(T)(T)T^{\alpha/2}-C_{12}(T)T\Big)\|\rho-\bar{\rho}\|_{C^1[0,T]} \\
    &\hspace{2em}+ \Big(1-C_9(T) T^{\alpha/2} - C_{11}(T) T^{\alpha/2}-C_{12}(T) T\Big) \|w-\bar{w}\|_{C^{1,1/2}([0,1]\times[0,T])}\\
    &\hspace{2em}+ \Big(1-C_{11}(T)T^{\alpha/2} - C_{12}(
    T)T\Big)\|\lambda -\bar{\lambda}\|_{L^\infty(0,T)} \le 0.
    \end{split}
\end{equation*}
If we choose $T>0$ sufficiently small such that
$$1-C_9(T) T^{\alpha/2} - C_{10}(T)T^{1/2} - C_{11}(T)T^{\alpha/2} - C_{12}(T) T>0,$$
then $u=\bar{u}$, $\rho=\bar{\rho}$, $w=\bar{w}$, $\lambda=\bar{\lambda}$.
\end{proof}

\section{Numerical simulations}

\subsection{The steady-state case}
In Section 2, we derived an algebraic characterization of the optimal control strategy in the steady-state case, as given in \re{ss-oc} and \re{ss-ocs}. For the numerical simulations, the parameter settings are as follows: critical concentration of nutrients $\widetilde{\sigma}=0.25$, tumor aggressiveness parameter $\mu=0.5$, the balancing parameter in the objective functional $B=2$, and the maximum amount of inhibitor $M=1$. Applying Newton's method to the equation \re{ss-oc}, we compute the optimal amount of inhibitor is $m_*=0.3269$. Hence, by \re{ss-ocs}, the optimal control variable is calculated as
$$m^* = \min\{m_*,M\}=\min\{0.3269,1\}=0.3269.$$

We will next develop the optimality system for the steady-state case using similar ideas in deriving the optimality system \re{os} for the parabolic case. This will further validate the result in Section 2 and bridge the gap between the steady-state and parabolic cases. Furthermore, it is intriguing to see that the  optimality system of the steady-state case does not follow from \eqref{os} by removing the time derivatives and initial/terminal conditions.

In the steady-state case, we get rid of the time derivatives in both the state system \re{eqn1} -- \re{eqn4} and the sensitivity system \re{sys2}.
When deriving the adjoint operators $L^*_1$, $L^*_2$, and $\mathcal{B}^*$, we disregard the time derivative terms in \re{operatorL} and \re{operatorB}, and note that, in steady state
$$\int_0^1 \mu(u-\widetilde{\sigma})\dif \xi = \frac{\rho'}{\rho}=0.$$
Thus, we have
\begin{equation*}
\begin{split}
\langle \begin{pmatrix}
        L_1 v \\ L_2 \eta
    \end{pmatrix}, \begin{pmatrix}
        w \\ \lambda 
    \end{pmatrix}\rangle &= \int_0^1 w\Big( - \frac{1}{\rho^2} v_{\xi\xi}  +(1+m)v\Big)\dif \xi  + \lambda\Big(- \eta \int_0^1 \mu (u-\widetilde{\sigma})\dif \xi\Big)\\
    &=\int_0^1 v\Big(-\frac{1}{\rho^2}w_{\xi\xi} + (1+m)w\Big)\dif \xi + \eta\cdot 0 \\
    &= \langle \begin{pmatrix}
        v \\ \eta
    \end{pmatrix}, \begin{pmatrix}
        L_1^* w \\ L_2^* \lambda 
    \end{pmatrix}\rangle,
    \end{split}
\end{equation*}
\begin{equation*}
\begin{split}
    \langle \mathcal{B}\begin{pmatrix}
        v \\ \eta
    \end{pmatrix},\begin{pmatrix}
        w \\ \lambda
    \end{pmatrix}\rangle &=  \int_0^1 w\Big( \frac{2\eta}{\rho^3}u_{\xi\xi}\Big) \dif \xi  +  \lambda \Big(-\rho \int_0^1 \mu v \dif \xi \Big) \\
    &=  \eta\Big( -\int_0^1 \frac{2}{\rho^3}w_\xi u_\xi \dif \xi \Big) +  \int_0^1 v\Big(-\mu \lambda \rho\Big)\dif \xi = \langle \begin{pmatrix}
        v \\ \eta
    \end{pmatrix},\mathcal{B}^*\begin{pmatrix}
        w \\ \lambda
    \end{pmatrix}\rangle
\end{split}
\end{equation*}
These calculations indicate that
\begin{equation*}
    \begin{pmatrix}
        L_1^* w \\ L_2^* \lambda 
    \end{pmatrix} = \begin{pmatrix}
        -\frac{1}{\rho^2}w_{\xi\xi} + (1+m)w \\ 0
    \end{pmatrix} \quad \text{ and } \quad \mathcal{B}^*\begin{pmatrix}
        w \\ \lambda
    \end{pmatrix} = \begin{pmatrix}
        -\mu\lambda\rho \\ -\int_0^1 \frac{2}{\rho^3}w_\xi u_\xi \dif \xi
    \end{pmatrix}
\end{equation*}
As a result, the optimality system for the steady-state elliptic PDEs is to find the unknown functions $u(\xi)$ and $w(\xi)$, and unknown constants $\rho$, $\lambda$, and $m^*$ such that
{\allowdisplaybreaks\begin{align*}
        &- \frac{u_{\xi\xi}}{\rho^2}  + (1+m^*)u  =0\hspace{2em} &0<\xi<1,\\
        & u = 1 \hspace{2em} &\xi = 1, \\
        & u_\xi = 0 \hspace{2em} &\xi=0,\\
        & 0=\int_0^1 \mu(u-\widetilde{\sigma}) \dif \xi \\
        &-\frac{w_{\xi\xi}}{\rho^2} + (1+m^*)w = \mu\lambda \rho
         \hspace{1em} &0<\xi<1,\\
        & w = 0 \hspace{2em} &\xi = 1, \\
        & w_\xi = 0 \hspace{2em} &\xi=0,\\
        & 0=   1+\frac{2}{\rho^3}\int_0^1 w_\xi u_\xi\dif \xi,\\
        & m^* = \min\Big\{\max\Big\{0,\frac{\int_0^1 w(\xi) u(\xi) \dif \xi}{2B} \Big\}, M\Big\}
    \end{align*}}
    To solve the above system, it is worth noting that, for a given constant $m$, the unknown function $u(\xi)$ can be explicitly solved as
\begin{equation*}
\label{special1}
    u(\xi) = \frac{\cosh(\sqrt{1+m}\rho \xi)}{\cosh(\sqrt{1+m}\rho)},
\end{equation*}
where $\rho$ is uniquely determined by
\begin{equation*}
    \frac{\tanh(\sqrt{1+m}\rho)}{\sqrt{1+m}\rho} = \widetilde{\sigma}.
\end{equation*}
Similarly, since the right-hand-side of the equation for $w$, given by $\mu\lambda\rho$, is a constant, $w$ can be explicitly solved as
\begin{equation*}
    \label{special2}
    w(\xi) = \frac{\mu \lambda \rho}{1+m}\Big(1- \frac{\cosh(\sqrt{1+m}\rho \xi)}{\cosh(\sqrt{1+m}\rho )}\Big),
\end{equation*}
where $\lambda$ is determined by the integral equation $\int_0^1 w_\xi u_\xi \dif \xi = -\frac{\rho^3}{2}$, leading to
$$\frac{1}{\mu\lambda} = \frac{\tanh(\sqrt{1+m}\rho)}{\sqrt{1+m}\rho}-\frac{1}{\cosh^2(\sqrt{1+m}\rho)}.$$
Therefore, for the steady-state case, we use the following algorithm to solve the optimality system and determine the optimal control $m^*$:

\noindent\rule{\textwidth}{1pt}
\begin{algorithm}[H]
\caption{The Steady-state Case}
\label{alg-ss}
\begin{algorithmic}
\STATE {Given $\widetilde{\sigma}$, $\mu$, $B$, $M$, and $TOL>0$, choose an initial constant $m_0 \in [0,M]$.}
\FOR{$i=0,1,2,\cdots$}
\STATE Compute $\rho_i$ from the equation $\frac{\tanh(\sqrt{1+m_i}\rho_i)}{\sqrt{1+m_i}\rho_i} = \widetilde{\sigma}$ using Newton's method;
\STATE Discretize $\xi$ within $[0,1]$ with a constant step size, and let $u_i(\xi_j) = \frac{\cosh(\sqrt{1+m_i}\rho_i \xi_j)}{\cosh(\sqrt{1+m_i}\rho_i)}$
\STATE Let $\lambda_i = \big(\frac{\tanh(\sqrt{1+m_i}\rho_i)}{\sqrt{1+m_i}\rho_i}-\frac{1}{\cosh^2(\sqrt{1+m_i}\rho_i)}\big)^{-1}\big/\mu$ and $w_i(\xi_j) = \frac{\mu \lambda_i \rho_i}{1+m_i}\Big(1-\frac{\cosh(\sqrt{1+m_i}\rho_i \xi_j)}{\cosh(\sqrt{1+m_i}\rho_i )}\Big)$;
\STATE Update the control variable $m_{i+1} = \min\Big\{\max\Big\{0,\frac{\int_0^1 w_i u_i \dif \xi}{2B} \Big\},M\Big\}$;
\STATE If $|m_{i+1}-m_i| < TOL$, stop the algorithm;
\ENDFOR
\end{algorithmic}
\end{algorithm}
\noindent\rule{\textwidth}{1pt}

In Figure \ref{fig:ss}, we implement Algorithm \ref{alg-ss} using a tolerance of $TOL=10^{-5}$ and two different initial guesses for $m_0$. In both scenarios, the optimized control $m$ quickly converges to a constant value, closely approximating $m^* =0.3269$. This outcome aligns with the result derived using direct computations in Section 2.

\begin{figure}[ht]
\centering
\begin{minipage}{0.48\textwidth}
\centering
\includegraphics[height=1.5in]{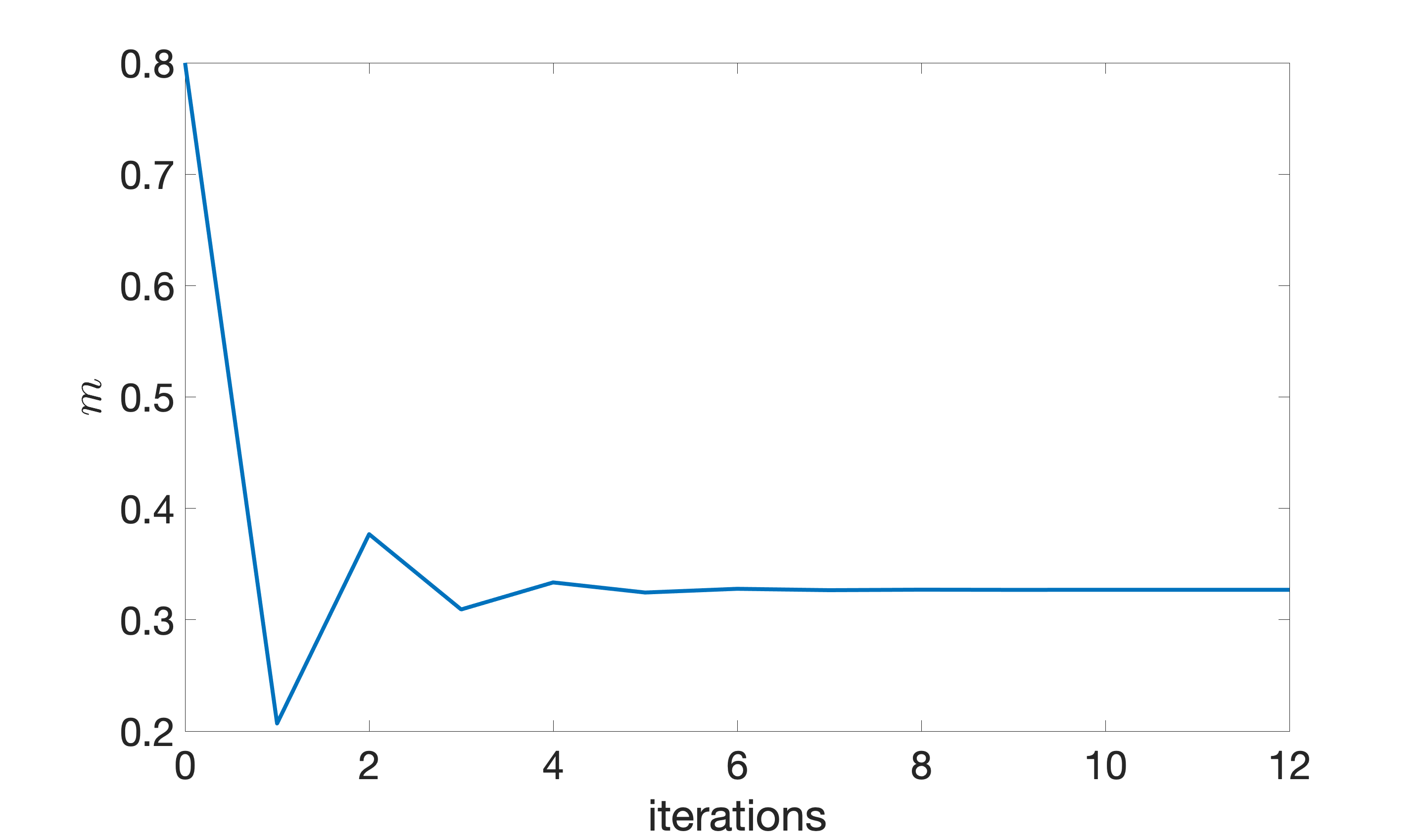}
\end{minipage}
\begin{minipage}{0.48\textwidth}
\centering
\includegraphics[height=1.5in]{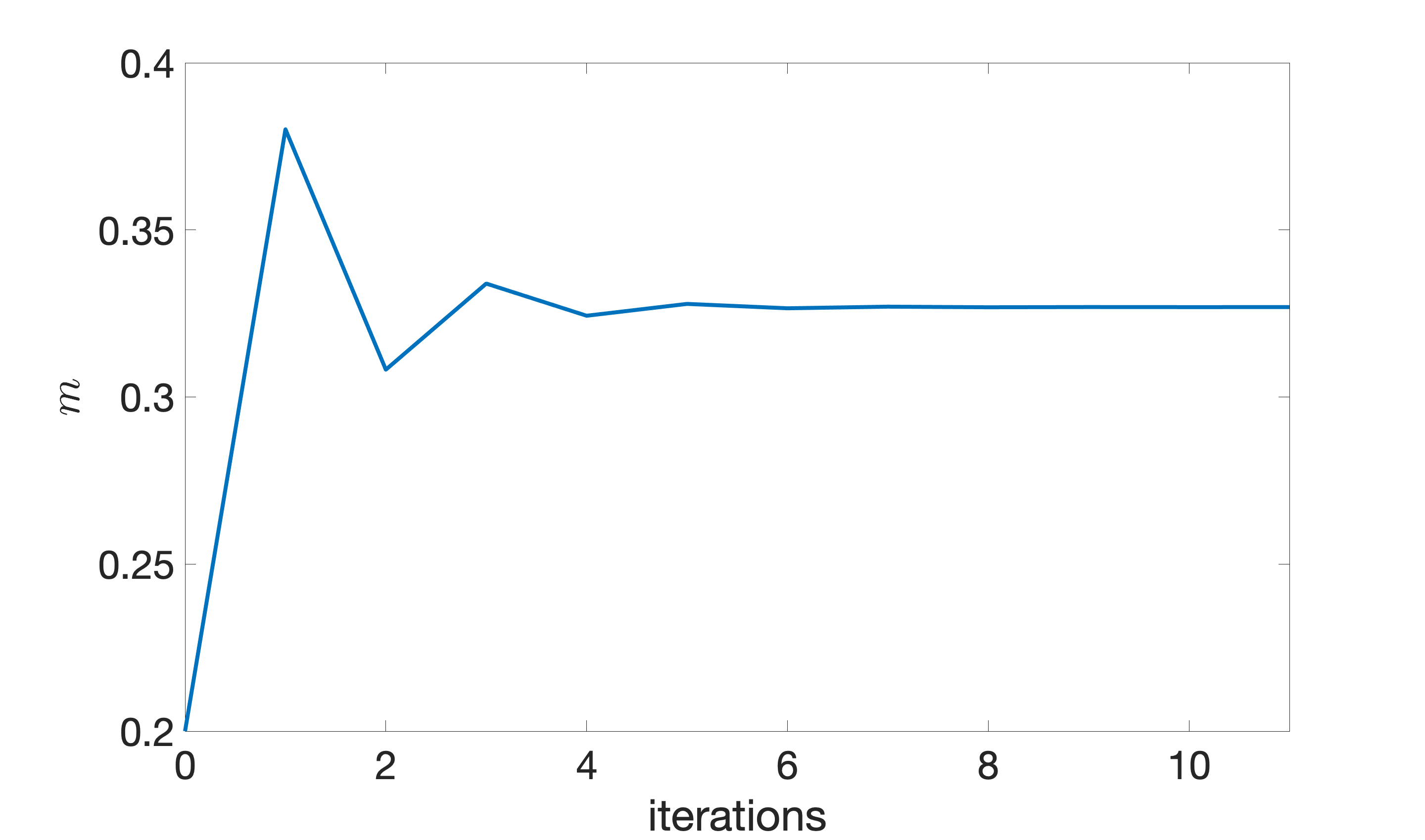}
\end{minipage}
\caption{The convergence of the optimal control in the steady-state case for $\widetilde{\sigma}=0.25$,$\mu=0.5$,$B=2$,$M=1$,and $TOL=10^{-5}$. The left figure shows the convergence with an initial $m_0=0.8$, and the right figure with an initial $m_0=0.2$.}
\label{fig:ss}
\end{figure}

\subsection{{Parabolic} case}
We employ the Forward-Backward Sweep Method \cite{hackbusch1978numerical,lenhart2007optimal} to solve the optimality system \re{os} and determine the optimal control.

\noindent\rule{\textwidth}{1pt}
\begin{algorithm}[H]
\caption{Forward-Backward Sweep Method}
\label{alg}
\begin{algorithmic}
\STATE {Given $\widetilde{\sigma}$, $\mu$, $B$, $M$, initial conditions $\rho_0$ and $u_0$, and $TOL>0$. Choose an initial guess $m_0\in U_M$.}
\FOR{$i=0,1,2,\cdots$}
\STATE Compute $(u_i,\rho_i)=(u^{m_i},\rho^{m_i})$ using a forward-in-time finite difference scheme;
\STATE Compute $(w_i,\lambda_i)=(w^{m_i},\lambda^{m_i})$ using a backward-in-time finite difference scheme;
\STATE Update the control variable $m_{i+1} = \min\Big\{\max\Big\{0,\frac{\int_0^1 w_i u_i \dif \xi}{2B}\Big\},M\Big\}$;
\STATE If $\|m_{i+1}-m_i\|_{L^\infty(0,T)} < TOL$, stop the algorithm;
\ENDFOR
\end{algorithmic}
\end{algorithm}
\noindent\rule{\textwidth}{1pt}

The results of applying Algorithm \ref{alg} to solve the complete optimality system \re{os} are shown in Figures \ref{fig:result1} -- \ref{fig:result4}. Although Theorem \ref{thm:os} necessitates a small final time $T$ to ensure the uniqueness of the solution, the simulations, in general, allow for a longer time interval. We solve the optimality system over the time interval $[0,5].$ With the same set of parameter values, a constant initial guess $m_0(t)=0.35$ is used in Figure \ref{fig:result1}, while an oscillatory initial guess $m_0(t)=0.35+0.1\cos(4\pi t)$ is employed in Figure \ref{fig:result2}. Despite starting with different initial values, both results converge to the same optimized control $m^*(t)$ and achieve the same objective value $J=11.5277$. The red dashed curve in the second row of Figures \ref{fig:result1} and \ref{fig:result2} shows $\rho(t)$ without control. We see that, with optimized control, the thickness of the tumor tissue decreases by about 20\% (from approximately 3.1 to around 2.5) by the final time $T=5$. It is also observed that the optimized treatment is at its upper bound at the beginning.

\begin{figure}
    \centering
    \includegraphics[width=0.88\textwidth]{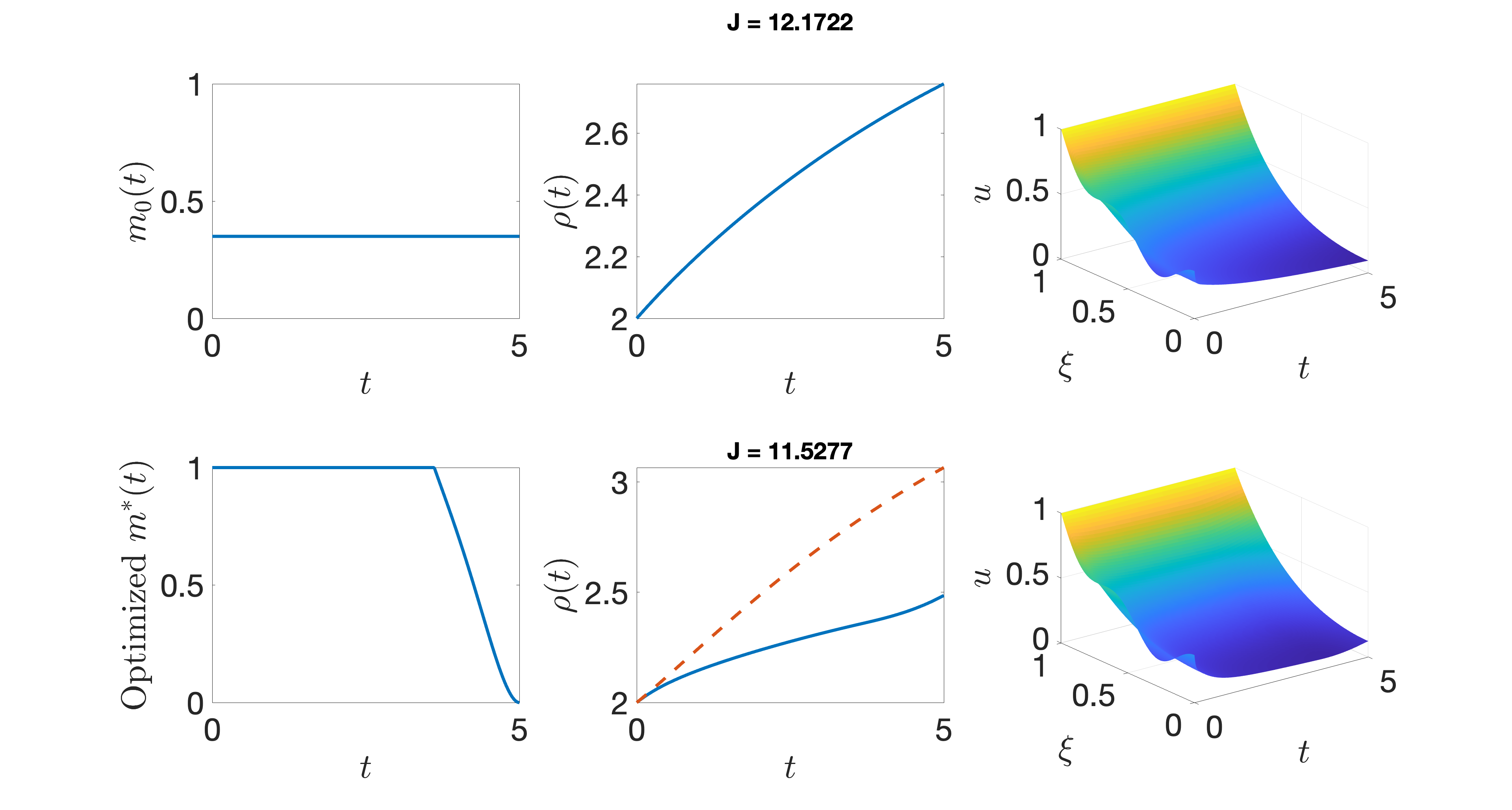}
    \caption{The chosen parameters are $\rho_0=2$, $u_0(\xi)=\frac{\cosh(\rho_0 \xi)}{\cosh(\rho_0)} + 0.1\cos(\frac72 \pi \xi)$, $\widetilde{\sigma}=0.25$, $B=0.05$, $M=1$, $\mu=0.5$, and $TOL=10^{-3}$. The first row shows the initial $m_0(t)=0.35$ and its corresponding $\rho(t)$, $u(\xi,t)$, and $J=12.1722$. The second row shows the optimized $m^*(t)$ and its corresponding $\rho(t)$, $u(\xi,t)$, and $J=11.5277$. The red dashed curve in the second row represents $\rho(t)$ without any control.}
    \label{fig:result1}
\end{figure}

\begin{figure}
    \centering
    \includegraphics[width=0.88\textwidth]{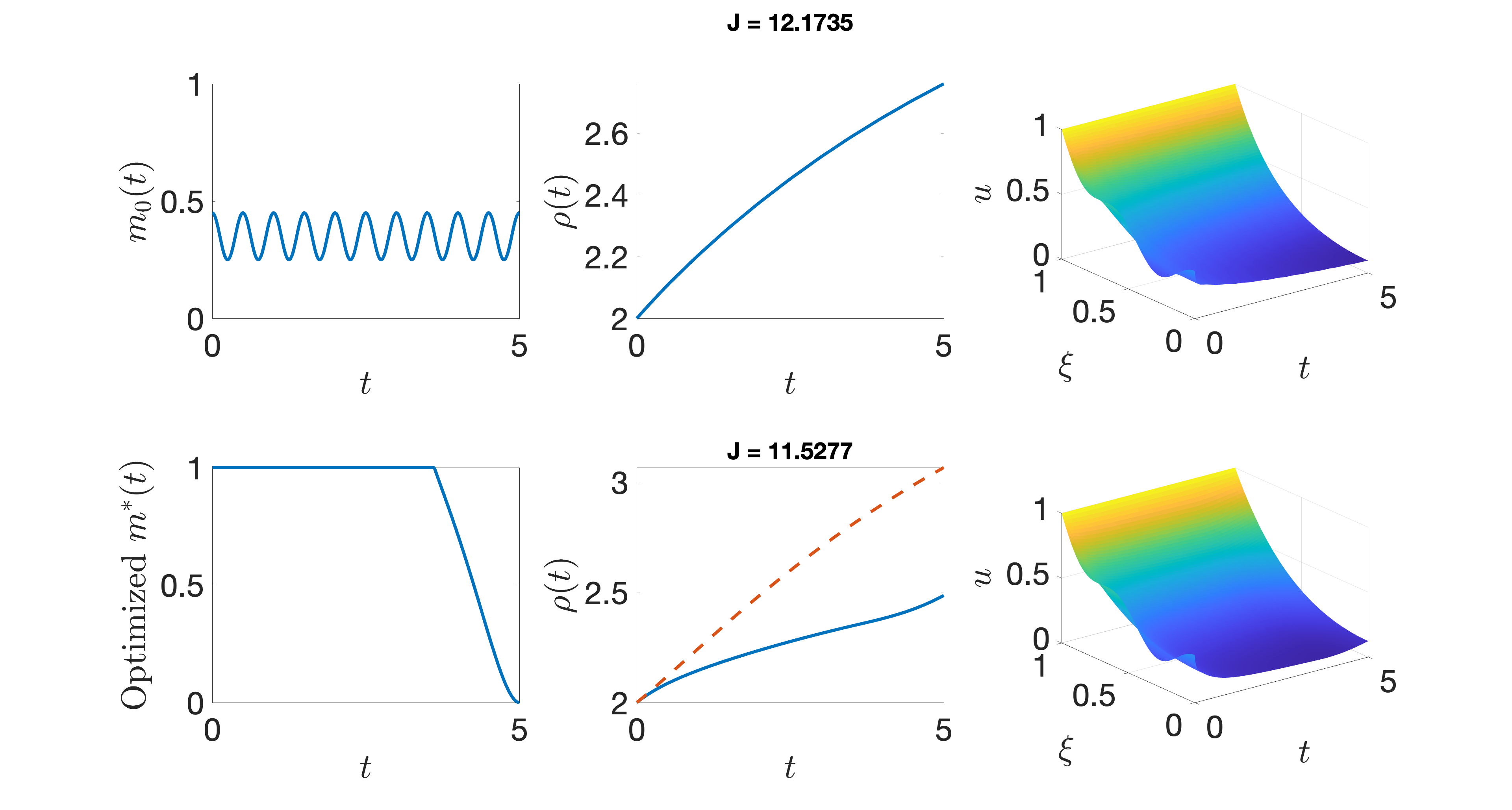}
    \caption{The chosen parameters are $\rho_0=2$, $u_0(\xi)=\frac{\cosh(\rho_0 \xi)}{\cosh(\rho_0)} + 0.1\cos(\frac72 \pi \xi)$, $\widetilde{\sigma}=0.25$, $B=0.05$, $M=1$, $\mu=0.5$, and $TOL=10^{-3}$. The first row shows the initial $m_0(t)=0.35+0.1\cos(4\pi t)$ and its corresponding $\rho(t)$, $u(\xi,t)$, and $J=12.1735$. The second row shows the optimized $m^*(t)$ and its corresponding $\rho(t)$, $u(\xi,t)$, and $J=11.5277$. The red dashed curve in the second row represents $\rho(t)$ without any control.}
    \label{fig:result2}
\end{figure}

In Figure \ref{fig:result3}, we set $\widetilde{\sigma}=0.75$ while keeping all other parameters unchanged. This higher $\widetilde{\sigma}$ value leads to a reduction in tumor thickness as time evolves. 
As illustrated in the middle figure, the solid curve, which represents tumor thickness with optimized control, shows a steeper decline compared to the dashed curve. It shows that, with the optimized control applied, the reduction in tumor thickness is markedly enhanced. When comparing the optimized control strategies in Figures \ref{fig:result1} and \ref{fig:result3}, we note that the amount of inhibitor, with the higher $\widetilde{\sigma}$ value, is administered at the maximum level for a shorter period of time before decreasing in a continuous manner. This observation is consistent with the steady-state case, as discussed in Section 2, which showed that an increase in $\widetilde{\sigma}$ results in a decrease in the optimal control.

Using the same parameter values as in Figure \ref{fig:result1} but increasing $B$ from 0.05 to 0.5, the results are shown in Figure \ref{fig:result4}. It is noted that the level of optimized control decreases as $B$ increases. Mathematically, this is due to the fact that the control intensity is inversely related to the balancing parameter $B$. If the side effects of the tumor growth inhibitor are more severe, then the patient receives a smaller amount of treatment. These results, again, align well with the findings discussed in the steady-state case in Section 2.

\begin{figure}
    \centering
    \includegraphics[width=0.88\textwidth]{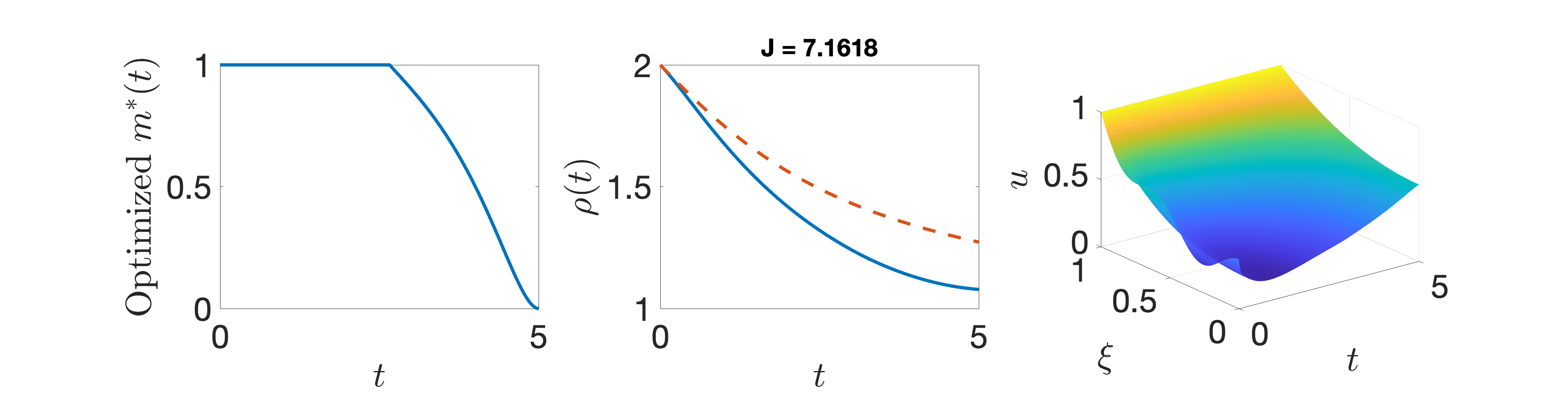}
    \caption{The optimized $m^*(t)$ and its corresponding $\rho(t)$ and $u(\xi,t)$. 
    The chosen parameters are $\rho_0=2$, $u_0(\xi)=\frac{\cosh(\rho_0 \xi)}{\cosh(\rho_0)} + 0.1\cos(\frac72 \pi \xi)$, $\widetilde{\sigma}=0.75$, $B=0.05$, $M=1$, $\mu=0.5$, and $TOL=10^{-3}$. The red dashed curve in the second figure represents $\rho(t)$ without any control.}
    \label{fig:result3}
\end{figure}

\begin{figure}
    \centering
    \includegraphics[width=0.88\textwidth]{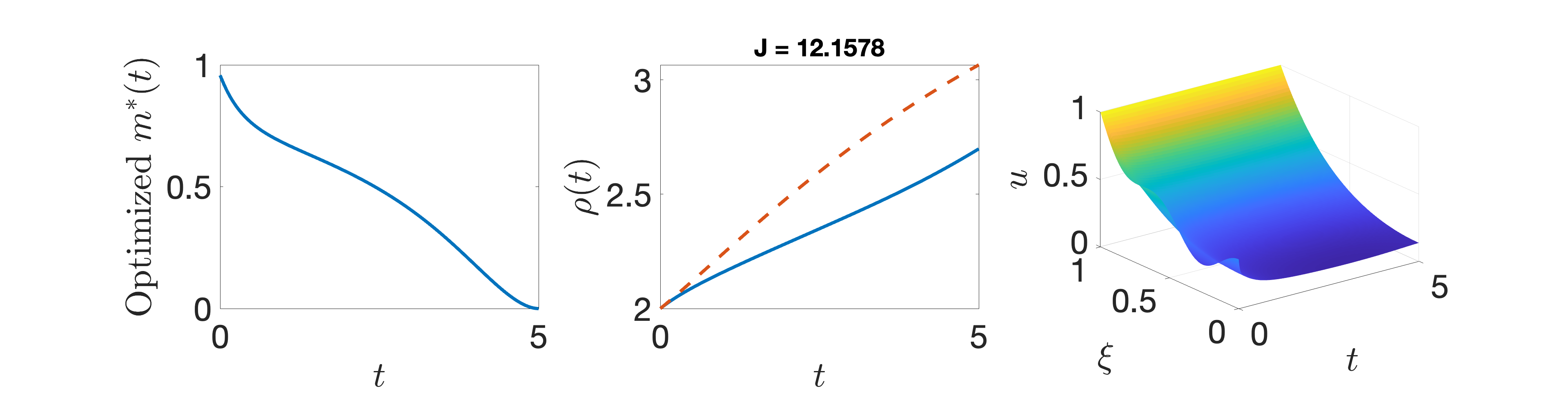}
    \caption{
    The optimized $m^*(t)$ and its corresponding $\rho(t)$ and $u(\xi,t)$. The chosen parameters are $\rho_0=2$, $u_0(\xi)=\frac{\cosh(\rho_0 \xi)}{\cosh(\rho_0)} + 0.1\cos(\frac72 \pi \xi)$, $\widetilde{\sigma}=0.25$, $B=0.5$, $M=1$, $\mu=0.5$, and $TOL=10^{-3}$. The red dashed curve in the second figure represents $\rho(t)$ without any control.}
    \label{fig:result4}
\end{figure}

\section{Discussion}
Free boundary models are commonly used to describe tumor tissue growth. To find the optimal amount of tumor growth inhibitor in a multilayered tumor growth model, we have developed a theoretical framework to tackle an optimization problem based on a free boundary PDE model. 
It has been proved that the optimization problem has a unique solution and can be characterized by the solution of the optimality system. {In the steady-state case, we found an algebraic characterization of the optimizer via direct computations in Section 2. We  validated the result by finding and solving the optimality system of the steady-state elliptic PDEs. For the parabolic case, we employed the forward-backward sweep method, a strategy commonly used in optimal control problems involving parabolic PDE systems, to solve the complete optimality system and determine the optimal control.}

There are a variety of potential extensions and directions for future research. One important direction is to incorporate the control variable into the boundary condition \re{e2} and investigate the corresponding optimal control problem. The boundary condition \re{e2} represents a constant nutrient supply to the tumor via blood vessels. In the context of tumor treatment, a widely recognized therapeutic method is cancer starvation therapy, which aims to block nutrient flow and suppress tumor growth \cite{selwan2016attacking,yu2019advances}. Therefore, modifying the boundary condition to include control could offer a more realistic strategy. Another potential direction is to explore the optimal control problem of the model \re{e1} -- \re{e9}. The challenge here lies in the inclusion of the $p$ variable and its boundary condition \re{e5}. In fact, due to this boundary condition, we only have local well-posedness, not global well-posedness, for the system \cite{cui2009well}. These optimization problems remain open and will be the focus of our future work.

\bigskip
\bibliography{ref}

\end{document}